\documentclass[11pt,oneside,reqno]{amsart}

\usepackage[margin=1in]{geometry}
\usepackage{amsmath,amssymb,amsthm,textcomp}
\usepackage{amsfonts,graphicx}
\usepackage{amsaddr}
\usepackage[mathscr]{eucal}
\usepackage{pgfplots}
\pgfplotsset{compat=1.15}
\usepackage{subfig,tikz}
\usepackage{array}
\usepackage{booktabs}
\usepackage{multicol}
\usepackage{multirow}
\usepackage[utf8]{inputenc}
\usepackage{bm}

\usepackage{color}
 
\usepackage{float}
\usepackage{hyperref}
 \hypersetup{colorlinks,citecolor=blue}
\hypersetup{colorlinks=true,linkcolor=red,filecolor=magenta,    urlcolor=blue}

\usepackage{pslatex}
\usepackage{array}
\usepackage[square,numbers]{natbib}
\allowdisplaybreaks[1]
\theoremstyle{definition}

\newcommand{\ncom}{\newcommand}

\ncom{\integ}[4]{\int_{#1}^{#2}\,{#3}\,d{#4}}
\ncom{\vspan}[1]{{{\rm\,span}\{ #1 \}}}
\ncom{\dm}[1]{ {\displaystyle{#1} } }
\ncom{\ri}[1]{{#1} \index{#1}}

\newtheorem{theorem}{\bf Theorem}[section]
\newtheorem{remark}{\bf Remark}[section]
\newtheorem{proposition}{Proposition}[section]
\newtheorem{lemma}{Lemma}[section]
\newtheorem{corollary}{Corollary}[section]

\newtheorem{definition}{Definition}[section]
\newtheoremstyle
{remarkstyle}
{}
{11pt}
{}
{}
{\bfseries}
{:}
{     }
{\thmname{#1} \thmnumber{#2} }

\theoremstyle{remarkstyle}

\newcommand{\ceil}[1]{\lceil{#1}\rceil}

\DeclareFontFamily{U}{BOONDOX-calo}{\skewchar\font=45 }
\DeclareFontShape{U}{BOONDOX-calo}{m}{n}{
  <-> s*[1.05] BOONDOX-r-calo}{}
\DeclareFontShape{U}{BOONDOX-calo}{b}{n}{
  <-> s*[1.05] BOONDOX-b-calo}{}
\DeclareMathAlphabet{\mathcalboondox}{U}{BOONDOX-calo}{m}{n}
\SetMathAlphabet{\mathcalboondox}{bold}{U}{BOONDOX-calo}{b}{n}
\DeclareMathAlphabet{\mathbcalboondox}{U}{BOONDOX-calo}{b}{n}

\begin{document}
\color{black}       
\title{Generalized Space Time Fractional Skellam Process}



\author{Kartik Tathe$^1$ and Sayan Ghosh$^2$}
\address{Department of Mathematics, Birla Institute of Technology and Science, Pilani, Hyderabad Campus, Jawahar Nagar, Kapra Mandal, Medchal District, Telangana 500078, India.}
\thanks{$^2$Corresponding author}
\email{$^1$kartikvtathe@gmail.com,$^2$sayan@hyderabad.bits-pilani.ac.in}





\begin{abstract}
This paper introduces the Generalized Space-Time Fractional Skellam Process (GSTFSP) and the Generalized Space Fractional Skellam Process (GSFSP). We investigate their distributional properties including the probability generating function (p.g.f.), probability mass function (p.m.f.), fractional moments, mean, variance, and covariance. We establish recurrence relations for the state probabilities of GSFSP and related processes, and also study the corresponding increment process along with its fractional version. Furthermore, we obtain the transition probabilities, arrival times and first passage times of these processes. The asymptotic behavior of the tail probabilities is analyzed, and limiting distributions as well as infinite divisibility of GSTFSP and GSFSP are studied. We establish the unique weighted sum representation and the martingale characterization for these processes. Then we introduce the running average processes of GSFSP and its special cases, and obtain their compound Poisson representations. Finally, the simulated sample paths for GSTFSP are plotted.
\end{abstract}

\maketitle


\section{Introduction}
The Poisson process continues to be a significant area of research due to its mathematical importance and broad applicability across various disciplines. Numerous generalizations of the Poisson process have been explored in the literature, particularly fractional generalizations, which have been developed through different approaches. One such approach involves considering renewal processes with inter-arrival times governed by Mittag-Leffler distributions (see \citet{Mainardi2004}). Another method replaces the time derivative in the governing equations of state probabilities with a fractional derivative (see \citet{Beghin2009}). In the case of the Caputo–Djrbashian fractional derivative, the two formulations yield an identical fractional version of the Poisson process. In recent years, random time-changed stochastic processes have been widely studied, where fractionality is introduced by altering the time flow of a process using a subordinator. Fractional extensions of the Poisson process have important applications in queueing theory, finance, and fractional quantum mechanics. They are also employed in modeling extreme events, such as earthquakes and tsunamis, where memory effects and anomalous temporal dynamics are significant.

In 2012, \citet{Orsingher2012} introduced the space fractional Poisson process (SFPP) by introducing fractionality in the backward shift operator which operates on the state space of the process. Additionally, they proposed the space-time fractional Poisson process (STFPP), which generalizes both  the time fractional Poisson process (TFPP/FPP) (see \citet{LASKIN2003}) and the SFPP. The STFPP was formulated by incorporating the Caputo fractional derivative into the governing equations of the probability generating function (p.g.f.) of the SFPP.

To allow multiple arrivals at a time and to accommodate varying arrival rates (unlike a Poisson process), the generalized counting process (GCP) and its time fractional version - generalized fractional counting process (GFCP) were introduced by \citet{Crescenzo2016}, enhancing their applicability and flexibility to model such real-world scenarios. 

Let $0<\alpha\le 1$ and $0<\beta\le 1$ denote the time and space fractional indices, respectively, of a counting process. \citet{Katariaarxiv} introduced the generalized space time fractional counting process (GSTFCP) denoted by
\begin{equation}\label{definition.GSTFCP}
    M_{\beta}^{\alpha}(t):=M(D_\beta(Y_\alpha(t)))
\end{equation}
where $\{M(t)\}_{t\geq 0}$ is a generalized counting process (GCP), $\{D_\beta(t)\}_{t\geq 0}$ is an independent $\beta$-stable subordinator and $\{Y_\alpha(t)\}_{t\geq 0}$ is an independent $\alpha$-inverse stable subordinator. 

The integer-valued L\'{e}vy process obtained by taking the difference of two independent Poisson processes is called the Skellam process (see \citet{Skellam1946} and \citet{Irwin2018}). Various generalizations of the Skellam process have also been a topic of interest among researchers. The time-fractional version of the Skellam process was studied by \citet{Kerss2014}. Recently, \citet{Gupta2020} introduced the Skellam process of order $k$ and the space fractional Skellam process along with its tempered version. \citet{kataria2024} introduced the fractional Skellam process of order $k$. Moreover \citet{Kataria2022b} introduced the generalized fractional Skellam process (GFSP) considering the difference of two independent GFCPs.  

The applications of Skellam process span various areas, such as modeling of intensity difference of pixels in cameras (see \citet{Hwang2007}) and modeling the difference in number of goals of two opponent teams in a football game (see \citet{Karlis2008}). In particular, fractional variants of the Skellam process have been used for modeling high-frequency financial data. As shown by \citet{Kerss2014}, the fractional Skellam process (FSP) models upward and downward jumps in tick-by-tick stock price data more accurately than the Skellam process. However, in practice, the
jumps (price changes) may not consistently occur at constant rates and the magnitude of jumps  may
fluctuate depending on various factors. Such type of data may also arise in other situations like queuing
systems and traffic flow. These issues motivate us to propose and study a more flexible model than existing ones by incorporating fractionality in both space and time components along with multiple variable jump intensities. The time-fractional generalization allows non-stationary and dependent increments over time while the space-fractional generalization captures the dependence of a state probability on previous such probabilities for a fixed time instance.

In this article, we introduce the Skellam version of GSTFCP namely the generalized space-time fractional Skellam process (GSTFSP) which generalizes all Skellam processes available in the literature. We also study a special case GSFSP, which is the Skellam version of GSFCP. The remaining paper is organized as follows. In Section \ref{sec 2}, we present preliminary definitions and essential results. In Section \ref{sec 3}, we introduce GSTFSP and examine its distributional properties (p.m.f. and p.g.f.) along with integer and fractional order moments of GSTFSP and GSTFCP. We also provide the correct p.m.f.s of some Skellam processes discussed in the literature. In Section \ref{sec 4}, we investigate the recurrence relations satisfied by the p.m.f.s. of GSFSP, GSFCP and their special cases. The increment process of GSFSP and its fractional version are discussed in Section \ref{sec 5}. Section \ref{sec 6} provides the transition probabilities of GSTFSP and examines its $n^{th}$ arrival times and first passage times. Section \ref{sec 7} analyzes the asymptotic behavior of the tail probabilities for such processes. Section \ref{sec 8} focuses on limiting distributions and infinite divisibility properties of GSTFSP, GSTFCP and related processes. In Section \ref{sec 9}, we obtain the weighted sum representations of GSTFSP and its special cases and investigate the martingale characterizations of the processes considered. Section \ref{sec 11} examines the running average processes of GSFSP and GSFCP. In Section \ref{sec 12}, we conclude by plotting the simulated sample path of GSTFSP. Throughout the paper, we consider both possible versions of GSTFSP $-$ GSTFSP I and GSTFSP II (see definitions \ref{definition_gstfsp1} and \ref{definition_gstfsp2}).  

\section{Preliminaries} \label{sec 2}
In this section, we present some preliminary results and definitions which will be used later in the paper.

\subsection{Stable subordinator and its inverse}\label{stable subordinator} An $\alpha$-stable subordinator $\{D_{\alpha}(t)\}_{t\geq 0}$, $0<\alpha\le 1$, is a non-decreasing L\'{e}vy process. Its Laplace transform is given by $\mathbb{E}(e^{-sD_\alpha(t)})=e^{-ts^\alpha}$, $s>0$ and the associated Bernstein function is $f(s)=s^\alpha$. Its first passage time $\{Y_\alpha(t)\}_{t\geq 0}$ is called the inverse $\alpha$-stable subordinator and defined as  
\begin{equation*}
Y_\alpha(t):=\inf\{x\geq 0:D_{\alpha}(x)>t \}.
\end{equation*}
The density of $Y_\alpha(t)$ denoted by $h_\alpha(x,s)$ (see \citet{Meerschaert2011}) has the Laplace transform (see \citet{Meerschaert2013}) $\tilde{h}_\alpha(x,t)=s^{\alpha-1}e^{-xs^{\alpha}}.$
The mean and variance (see \citet{Leonenko2014}) of $Y_\alpha(t)$ are given by
\begin{align*}
\mathbb{E}(Y_{\alpha}(t))=\frac{t^\alpha}{\Gamma(\alpha+1)}, \quad
\mathbb{V}(Y_\alpha(t))=(\frac{2}{\Gamma(2\alpha+1)}-\frac{1}{\Gamma^2(\alpha+1)})t^{2\alpha}.
\end{align*}

\subsection{Tempered Stable Subordinator and its inverse}\label{TSS}
A tempered stable subordinator (TSS) $\{D_{\alpha,\theta}(t)\}_{t\ge 0}$ with stability index $0<\alpha\le 1$ and tempering parameter $\theta>0$ is obtained by exponentially tempering the distribution of $\alpha$-stable subordinator $\{D_\alpha(t)\}_{t\ge 0}$ (see \citet{Rosinski2007}). The Bernstein function associated with TSS is given by $f_2(s)=(\theta+s)^\alpha-\theta^\alpha$, $s>0$. The first hitting time of TSS $\{\mathcal{L}_{\alpha,\theta}(t)\}_{t\ge 0}$ known as the inverse TSS is defined as
\begin{equation*}
    \mathcal{L}_{\alpha,\theta}(t):=\inf\{x\ge 0: D_{\alpha,\theta}(x)>t\}.
\end{equation*}

\subsection{Caputo–Djrbashian fractional derivative}\label{caputo}
The Caputo–Djrbashian fractional derivative of a function $f$ defined on $[0,\infty)$ is given by (see \citet{Leonenko2017})
\begin{equation*}
    \partial_{t}^{\alpha}f(t)=\frac{d^\alpha}{dt^\alpha}f(t)=\frac{1}{\Gamma(1-\alpha)}\int_{0}^{t}\frac{df(\tau)}{d\tau}\frac{d\tau}{(t-\tau)^{\alpha}}, \quad 0<\alpha\le 1.
\end{equation*}
Its Laplace transform is 
\begin{equation*}
    \mathcal{L}\{\frac{d^\alpha}{dt^\alpha}f\}(s)=s^\alpha \mathcal{L}\{f\}(s)-s^{\alpha-1}f(0^+).
\end{equation*}



\subsection{Two-parameter Mittag-Leffler function}
The two-parameter Mittag-Leffler function is defined as (see \citet{Prajapati2012})
\begin{equation*}
E_{\alpha,\beta}(z)=\sum_{k=0}^{\infty}\frac{z^k}{\alpha k+\beta},\quad\alpha,\,\beta,\,z\in\mathbb{C},\,\, \text{the set of complex numbers with Re}(\alpha)>0.
\end{equation*}
For $\alpha=1$ and $\beta=1$, $E_{\alpha,\beta}(z)$ reduces to $\exp(z)$.

\subsection{Fox-Wright function}\label{Generalized Wright function} 
The Fox-Wright function is defined by (see \citet{Konovska2024})
\begin{equation*}
    {}_{p}\Psi_{q}\!\left[
\begin{array}{cccc}
(a_{1},A_{1}) & (a_{2},A_{2}) & \ldots & (a_{p},A_{p})\\
(b_{1},B_{1}) & (b_{2},B_{2}) & \ldots & (b_{q},B_{q})
\end{array}
\bigg| z
\right]
=\sum_{n=0}^{\infty}
\frac{\prod_{i=1}^{p}
\Gamma(a_{i}+A_{i}n)}{\prod_{j=1}^{q}
\Gamma(b_{j}+B_{j}n)}
\,\frac{z^{n}}{n!}.
\end{equation*}
with arbitrary positive parameters $A_i$ and $B_j$, complex $a_i$ and $b_j$, $1\le i\le p$, $1\le j \le q$, $ p \le q$ or $p = q + 1$.


\subsection{Lemma} (\citet{Xia2018}) If $\{X(t)\}_{t\ge 0}$ is a L\'{e}vy process and its Riemann integral is defined by $Y(t)=\int_{0}^{t}X(s)ds,$
then the characteristic function of $Y(t)$ is given by
\begin{equation*}
    \phi_{Y(t)}(u)=\mathbb{E}[e^{iuY(t)}]=e^{t(\int_{0}^{1}\log\phi_{X(1)}(tuz)dz)}, \quad u\in\mathbb{R}. \label{running average lemma}
\end{equation*}



\section{Generalized Space Time Fractional Skellam process}\label{sec 3}
In this section, we introduce two distinct variants of GSTFSP namely GSTFSP-I and GSTFSP-II, and study their distributional characteristics such as p.m.f., p.g.f. and moments.
\begin{definition}\label{definition_gstfsp1}

Let $\{M_{1}(t)\}_{t\ge 0}$ and 
$\{M_{2}(t)\}_{t\ge 0}$ be two independent 
GCPs with parameters $\lambda_{1},\ldots,\lambda_{k}$ and 
$\mu_{1},\ldots,\mu_{k}$, respectively, and denote 
$\Lambda=\sum_{j=1}^{k}\lambda_{j}$ and $T=\sum_{j=1}^{k}\mu_{j}$. 
Let $\{D_{\beta_1}(t)\}_{t\ge 0}$ and $\{D_{\beta_2}(t)\}_{t\ge 0}$ be 
two stable subordinators (with stability indices $\beta_1$ and $\beta_2$), and $\{Y_{\alpha_1}(t)\}_{t\ge 0}$ and 
$\{Y_{\alpha_2}(t)\}_{t\ge 0}$ be two inverse stable subordinators (with stability indices $\alpha_1$ and $\alpha_2$). 
Assume that all these subordinators are independent of each other 
and also independent of the GCPs. The 
generalized space-time fractional Skellam process of type~I 
(GSTFSP-I) is then defined as 
\begin{equation}\label{definition.GSTFSP-I}
    \mathcal{S}_{\beta_1,\beta_2}^{\alpha_1,\alpha_2}(t)
 := M_{1\alpha_1}^{\beta_1}(t)-M_{2\alpha_2}^{\beta_2}(t)=M_1(D_{\beta_1}(Y_{\alpha_1}(t))) 
 - M_2(D_{\beta_2}(Y_{\alpha_2}(t))).
\end{equation}
\end{definition}
\begin{definition}\label{definition_gstfsp2}
Let $\{\mathcal{S}(t)=M_1(t)-M_2(t)\}_{t\ge 0}$ be a generalized Skellam process (GSP), where $\{M_1(t)\}_{t\ge 0}$ and $\{M_2(t)\}_{t\ge 0}$ are independent GCPs. Let $\{D_{\beta}(t)\}_{t\ge 0}$ be a $\beta$-stable subordinator 
and $\{Y_{\alpha}(t)\}_{t\ge 0}$ an inverse $\alpha$-stable subordinator, 
independent of each other and of the GSP. The generalized 
space-time fractional Skellam process of type~II (GSTFSP-II) is 
then defined as
\begin{equation}\label{definition.GSTFSP-II}
    \mathcal{S}_{\beta}^{\alpha}(t)=\mathcal{S}(D_\beta(Y_\alpha(t))).
\end{equation}
\end{definition}

\subsection{Special cases of GSTFSP}\label{special cases_GSTFCP}
For $k=1$, GSTFSP-II reduces to STFSP-II, the Skellam version of STFPP $\{N_\beta^\alpha(t)\}_{t\geq 0}$ (see \citet{Orsingher2012}). When $\alpha=1$, GSTFSP-II reduces to GSFSP-II $\{\mathcal{S}_{\beta}(t)\}_{t\geq 0}$, the Skellam version of GSFCP $\{M_{\beta}(t)\}_{t\geq 0}$. For $\alpha=1$ and $k=1$, it further reduces to SFSP-II, the Skellam version of SFPP $\{N_\beta(t)\}_{t\geq 0}$ (see \citet{Orsingher2012}). Setting $\beta=1$ reduces GSTFSP-II to GFSP-II $\{\mathcal{S}^\alpha(t)\}_{t\geq 0}$, the Skellam version of GFCP (see \citet{KatariaGFCP}). For $k=1$, GFSP-II reduces to FSP-II $\{S^\alpha(t)\}_{t\geq 0}$, the Skellam version of TFPP (see \citet{LASKIN2003}) and for $\alpha=1$, it reduces to GSP $\{\mathcal{S}(t)\}_{t\geq 0}$, the Skellam version of GCP (see  \citet{Crescenzo2016}). Note that the GSP reduces to the Skellam process $\{S(t)\}_{t\geq 0}$ if $k=1$. Next for $\lambda_j=\lambda$ and $\mu_j=\mu$ where $1\le j \le k$, GSTFSP-II reduces to STFSPoK-II $\{S_{\beta}^{k,\alpha}(t)\}_{t\geq 0}$, the Skellam version of STFPPoK $\{N_{\beta}^{k,\alpha}(t)\}_{t\geq 0}$. For $\alpha=1$, STFSPoK-II further reduces to SFSPoK-II $\{S_{\beta}^{k}(t)\}_{t\geq 0}$, the Skellam version of SFPPoK $\{N_{\beta}^{k}(t)\}_{t\geq 0}$ and for $\beta=1$, it reduces to FSPoK-II $\{S^{k,\alpha}(t)\}_{t\geq 0}$ (see \citet{kataria2024}), the Skellam version of FPPoK $\{N^{k,\alpha}(t)\}_{t\geq 0}$ (see \citet{Gupta2023}). Finally, note that FSPoK-II simplifies to SPoK $\{S^{k}(t)\}_{t\geq 0}$, the Skellam version of PPoK $\{N^{k}(t)\}_{t\geq 0}$ for $\alpha=1$. The various special cases of GSTFSP-I can be obtained similarly.   

\subsection{Probability mass function}
Consider a GCP which has a weighted sum representation given by $\{M(t)=\sum_{j=1}^{k}jN_j(t)\}_{t\ge 0}$, where $\{N_j(t)\}_{t\ge 0}$ are $k$ independent Poisson processes with parameters $\lambda_1,\dots,\lambda_k$ (see \citet{Dhillon2024}). Its p.m.f. is given by 
\begin{equation*}
P(M(t)=n)=\sum_{\Omega(k,n)}e^{-\Lambda t}(\prod_{j=1}^{k}\frac{(\lambda_j t)^{x_j}}{x_j!})=\sum_{\Omega(k,n)}\prod_{j=1}^{k}\frac{\lambda_j^{x_j}}{x_j!}(-\partial_{\lambda_j})^{x_j}e^{-\lambda_j s},~n\ge 0,    
\end{equation*}
where $\Omega(k,n)= \{(x_1,x_2,...,x_k)|\sum_{j=1}^{k}jx_j=n, x_j\ge 0\}$ and $(\partial_{\lambda_j})^{x_j} = \frac{\partial^{x_j}}{\partial \lambda_j^{x_j}}$. So the p.m.f. of a GSTFCP is
\begin{align}
    P(M_{\beta}^{\alpha}(t)=n)=P(M(D_\beta(Y_\alpha(t)))=n)&=\int_{0}^{\infty}P(M(D_\beta(Y_\alpha(t)))=n|D_\beta(Y_\alpha(t))\in ds)P(D_\beta(Y_\alpha(t))\in ds)\nonumber\\
    &=\int_{0}^{\infty}P(M(s)=n)P(D_\beta(Y_\alpha(t))\in ds)\nonumber\\
    &=\sum_{\Omega(k,n)}\int_{0}^{\infty}\prod_{j=1}^{k}P(N_j(s)=x_j)P(D_\beta(Y_\alpha(t))\in ds)\nonumber\\
    &=\sum_{\Omega(k,n)}\prod_{j=1}^{k}\frac{\lambda_j^{x_j}}{x_j!}\int_{0}^{\infty}(-\partial_{\lambda_j})^{x_j}e^{-\lambda_j s}P(D_\beta(Y_\alpha(t))\in ds)\nonumber\\
    &=\sum_{\Omega(k,n)}\prod_{j=1}^{k}\frac{\lambda_j^{x_j}}{x_j!}(-\partial_{\lambda_j})^{x_j}\mathbb{E}(e^{-\lambda_j D_\beta(Y_\alpha(t))})\nonumber\\
    &=\sum_{\Omega(k,n)}\prod_{j=1}^{k}\frac{\lambda_j^{x_j}}{x_j!}(-\partial_{\lambda_j})^{x_j}E_{\alpha,1}(-\lambda_j^\beta t^\alpha),\label{pmf_gstfcp}
\end{align}
where the last step is obtained using Lemma $3.1$ of \citet{Beghin2014}.
\begin{remark}
    Note that the p.m.f. of the GSTFCP in \citet{Katariaarxiv} implicitly assumes that
    \begin{equation*}
        \prod_{j=1}^{k}(-\partial_{\lambda_j})^{x_j}E_{\alpha,1}(-\lambda_j^\beta t^\alpha)=(-\partial_{\Lambda})^{z_k}E_{\alpha,1}(-\Lambda^\beta t^\alpha),
    \end{equation*}
    where $\Lambda=\sum_{j=1}^{k}\lambda_j$ and $z_k=\sum_{j=1}^{k}x_j$. However, the above equation holds only when $k=1$ (corresponding to a STFPP) or $\alpha=\beta=1$ (corresponding to a GCP). So the correct p.m.f. of a GSTFCP is given by \eqref{pmf_gstfcp}. Similarly, the correct p.m.f.s of GFCP and GSFCP discussed in \citet{Katariaarxiv} can be obtained by substituting $\beta=1$ and $\alpha=1$ respectively in \eqref{pmf_gstfcp}.
\end{remark}

\begin{remark}
The p.m.f. of a GSTFCP can alternatively be expressed as follows:
\begin{align}\label{fox-wright_gstfcp}
 P(M_{\beta}^{\alpha}(t)=n)
 &=\sum_{\Omega(k,n)}\prod_{j=1}^{k}\frac{(-1)^{x_j}}{x_j!}\sum_{m=0}^{\infty}\frac{(-\lambda_j^\beta t^\alpha)^m}{\Gamma(\alpha m+1)}\frac{\Gamma(\beta m+1)}{\Gamma(\beta m+1-x_j)} \nonumber \\
 &=\sum_{\Omega(k,n)}\prod_{j=1}^{k}\frac{(-1)^{x_j}}{x_j!}{}_2\psi_2\left[\begin{array}{cc}(1,1)&(1,\beta) \\(1,\alpha)&(1-x_j,\beta)\end{array}\bigg|-\lambda_j^\beta t^\alpha\right], 
\end{align}
where ${}_p\psi_q(.)$ is the Fox-Wright function. Similar representations for the p.m.f.s of GFCP and GSFCP can be obtained by putting $\beta=1$ and $\alpha=1$ respectively in \eqref{fox-wright_gstfcp}.
\end{remark}
The following result will be used to obtain the p.m.f. of GSTFSP-I.
\begin{lemma}\label{Omega_set_equality_lemma}
For $k\ge 0$ and $n\in\mathbb{Z}$, we have
\begin{equation*}
\mathcal A(k,n)=
\begin{cases}
\displaystyle\bigcup_{m\ge0}\big(\Omega(k,m+n)\times\Omega'(k,m)\big), & n\ge0,\\
\displaystyle\bigcup_{m\ge0}\big(\Omega(k,m)\times\Omega'(k,m+|n|)\big), & n<0,
\end{cases}
\end{equation*} 
 where $\displaystyle\mathcal{A}(k,n)= \{(x_1,\dots,x_k,y_1,\dots,y_k)\mid\,x_j\ge 0,~y_j\ge 0,~\sum_{j=1}^{k}j(x_j-y_j)=n\}$,\\ $\displaystyle
    \Omega(k,m)= \{(x_1,\dots,x_k)\mid\sum_{j=1}^{k}jx_j=m,~x_j\ge 0\}$
and $\displaystyle\Omega'(k,m)=\{(y_1,\dots,y_k)\mid\sum_{j=1}^{k}jy_j=m,~y_j\ge 0\}$.
\end{lemma}
\begin{proof}
In this proof, we denote an element of $\Omega(k,\cdot)\times\Omega'(k,\cdot)$ by $({\bf x},{\bf y})=(x_1,\dots,x_k,y_1,\dots,y_k)$ where ${\bf x}=(x_1,\dots,x_k)\in\Omega(k,\cdot)$ and ${\bf y}=(y_1,\dots,y_k)\in\Omega'(k,\cdot)$ thus viewing each pair $({\bf x},{\bf y})$ as a single concatenated coordinate vector. 
First, consider the case $n\ge 0$.  
If $({\bf x},{\bf y})\in\mathcal A(k,n)$, then by definition $x_j,y_j\ge0$ and $\sum_{j=1}^k j(x_j-y_j)=n$.  
Setting $m=\sum_{j=1}^k jy_j$, we obtain $\sum_{j=1}^k jx_j=m+n$, which shows that  
${\bf x}\in\Omega(k,m+n)$ and ${\bf y}\in\Omega'(k,m)$ for some $m\ge 0$. Hence
\begin{equation}\label{omega_equality_subset1}
\mathcal A(k,n)\subseteq \bigcup_{m\ge0}\big(\Omega(k,m+n)\times\Omega'(k,m)\big).
\end{equation}
Conversely, if $({\bf x},{\bf y})\in\Omega(k,m+n)\times\Omega'(k,m)$ for some $m\ge0$, then  
$\sum_{j=1}^k jx_j=m+n$ and $\sum_{j=1}^k jy_j=m$, which implies that  
$\sum_{j=1}^k j(x_j-y_j)=n$ so $({\bf x},{\bf y})\in\mathcal A(k,n)$. Therefore,
\begin{equation}\label{omega_equality_subset2}
\bigcup_{m\ge0}\big(\Omega(k,m+n)\times\Omega'(k,m)\big)\subseteq \mathcal A(k,n).
\end{equation}
Combining \eqref{omega_equality_subset1} and \eqref{omega_equality_subset2}, we get the desired result for $n\ge0$. Now consider $n<0$.  
If $({\bf x},{\bf y})\in\mathcal A(k,n)$, then $\sum_{j=1}^k j(x_j-y_j)=n=-|n|$.  
Defining $m=\sum_{j=1}^k jx_j$, we obtain $\sum_{j=1}^k jy_j=m+|n|$, so  
${\bf x}\in\Omega(k,m)$ and ${\bf y}\in\Omega'(k,m+|n|)$. Therefore
\begin{equation}\label{omega_equality_subset3}
\mathcal A(k,n)\subseteq \bigcup_{m\ge0}\big(\Omega(k,m)\times\Omega'(k,m+|n|)\big).
\end{equation}
Conversely, if $({\bf x},{\bf y})\in\Omega(k,m)\times\Omega'(k,m+|n|)$, then  
$\sum_{j=1}^k jx_j=m$ and $\sum_{j=1}^k jy_j=m+|n|$, which gives  
$\sum_{j=1}^k j(x_j-y_j)=n$, so $({\bf x},{\bf y})\in\mathcal A(k,n)$. Hence,
\begin{equation}\label{omega_equality_subset4}
\bigcup_{m\ge0}\big(\Omega(k,m)\times\Omega'(k,m+|n|)\big)\subseteq \mathcal A(k,n).
\end{equation}
Combining \eqref{omega_equality_subset3} and \eqref{omega_equality_subset4}, we get the desired result for $n<0$.
\end{proof}
\begin{theorem}\label{p.m.f..gstfsp1}
    The p.m.f. of GSTFSP-I is given by
    \begin{align*}
        P(\mathcal{S}_{\beta_1,\beta_2}^{\alpha_1,\alpha_2}(t)=n)&=\sum_{\mathcal{A}(k,n)}\prod_{j=1}^{k}\frac{\lambda_j^{x_j}\mu_j^{y_j}}{x_j!y_j!}((-\partial_{\lambda_j})^{x_j}E_{\alpha_1,1}(-\lambda_j^{\beta_1} t^{\alpha_1}))((-\partial_{\mu_j})^{y_j}E_{\alpha_2,1}(-\mu_j^{\beta_2} t^{\alpha_2}))\quad\forall~n\in \mathbb{Z},
    \end{align*}
   where $\displaystyle\mathcal{A}(k,n)=\{(x_1,\dots,x_k,y_1,\dots,y_k)\mid x_j\ge 0,~y_j\ge 0,\sum_{j=1}^{k}j(x_j-y_j)=n\}$, $\displaystyle(\partial_{\lambda_j})^{x_j}=\frac{\partial^{x_j}}{\partial \lambda_j^{x_j}}$, $\displaystyle(\partial_{\mu_j})^{y_j} = \frac{\partial^{y_j}}{\partial \mu_j^{y_j}}$.
\end{theorem}
\begin{proof} 
Using Definition \ref{definition_gstfsp1} and the p.m.f. \eqref{pmf_gstfcp}, we obtain
\begin{align*}
P(\mathcal{S}_{\beta_1,\beta_2}^{\alpha_1,\alpha_2}(t)=n)&=P(M_{1\beta_1}^{\alpha_1}(t)-M_{2\beta_2}^{\alpha_2}(t)=n)\mathbb{I}_{\{n\ge 0\}}+P(M_{1\beta_1}^{\alpha_1}(t)-M_{2\beta_2}^{\alpha_2}(t)=n)\mathbb{I}_{\{n< 0\}}\\
&=\sum_{m=0}^{\infty}[P(M_{1\beta_1}^{\alpha_1}(t)=m+n)P(M_{2\beta_2}^{\alpha_2}(t)=m)\mathbb{I}_{\{n\ge 0\}}+P(M_{2\beta_2}^{\alpha_2}(t)=m+|n|)P(M_{1\beta_1}^{\alpha_1}(t)=m)\mathbb{I}_{\{n< 0\}}]\\
&=\sum_{m=0}^{\infty}[\sum_{\Omega(k,m+n)}\sum_{\Omega'(k,m)}(\prod_{j=1}^{k}\frac{\lambda_j^{x_j}}{x_j!}(-\partial_{\lambda_j})^{x_j}E_{\alpha_1,1}(-\lambda_j^{\beta_1} t^{\alpha_1}))(\prod_{j=1}^{k}\frac{\mu_j^{y_j}}{y_j!}(-\partial_{\mu_j})^{y_j}E_{\alpha_2,1}(-\mu_j^{\beta_2} t^{\alpha_2}))\mathbb{I}_{\{n\ge 0\}}\\
&+\sum_{\Omega'(k,m+|n|)}\sum_{\Omega(k,m)}(\prod_{j=1}^{k}\frac{\mu_j^{y_j}}{y_j!}(-\partial_{\mu_j})^{y_j}E_{\alpha_2,1}(-\mu_j^{\beta_2} t^{\alpha_2}))(\prod_{j=1}^{k}\frac{\lambda_j^{x_j}}{x_j!}(-\partial_{\lambda_j})^{x_j}E_{\alpha_1,1}(-\lambda_j^{\beta_1} t^{\alpha_1}))\mathbb{I}_{\{n< 0\}}]\\
\end{align*}
Note that $\displaystyle\mathcal{A}(k,n)=\bigcup_{m\ge0}\{({\bf x},{\bf y})\big|{\bf x}=(x_1,\dots,x_k)\in\Omega(k,m+n)~\text{and}~{\bf y}=(y_1,\dots,y_k)\in\Omega'(k,m)\}$  for $n\ge 0$ and $\displaystyle\mathcal{A}(k,n)=\bigcup_{m\ge0}\{({\bf x},{\bf y})\big|{\bf x}=(x_1,\dots,x_k)\in\Omega(k,m)~\text{and}~{\bf y}= (y_1,\dots,y_k)\in\Omega'(k,m+|n|)\}$, for $n< 0$. Therefore using Lemma \ref{Omega_set_equality_lemma}, we can rewrite the p.m.f. as
\begin{align*}
P(\mathcal{S}_{\beta_1,\beta_2}^{\alpha_1,\alpha_2}(t)=n)&=\sum_{\mathcal{A}(k,n)}\prod_{j=1}^{k}\frac{\lambda_j^{x_j}\mu_j^{y_j}}{x_j!y_j!}((-\partial_{\lambda_j})^{x_j}E_{\alpha_1,1}(-\lambda_j^{\beta_1} t^{\alpha_1}))((-\partial_{\mu_j})^{y_j}E_{\alpha_2,1}(-\mu_j^{\beta_2} t^{\alpha_2}))\quad\forall~n\in\mathbb{Z}.
\end{align*}


\end{proof}
Substituting $\alpha_1=\alpha_2=1$ and $\beta_1=\beta_2=1$ in Theorem \ref{p.m.f..gstfsp1}, we obtain the p.m.f. of the GSP as follows:
\begin{equation}\label{pmf_GSP}
         P(\mathcal{S}(t)=n)=\sum_{\mathcal{A}(k,n)}e^{-t(\Lambda+T)}\prod_{j=1}^{k}\frac{\lambda_j^{x_j}\mu_j^{y_j}}{x_j!y_j!}t^{x_j+y_j}.
\end{equation}
\begin{remark}
Expanding the Mittag-Leffler functions, the p.m.f. of a GSTFSP-I can alternatively be expressed in terms of Fox-Wright functions as follows:
\begin{align}\label{fox-wright_gstfsp1}
 P(\mathcal{S}_{\beta_1,\beta_2}^{\alpha_1,\alpha_2}(t)=n)&=\sum_{\mathcal{A}(k,n)}\prod_{j=1}^{k}\frac{(-1)^{x_j+y_j}}{x_j!y_j!}\sum_{m=0}^{\infty}\frac{(-\lambda_j^{\beta_1} t^{\alpha_1})^m}{\Gamma(\alpha_1 m+1)}\frac{\Gamma(\beta_1 m+1)}{\Gamma(\beta_1 m+1-x_j)}\sum_{m=0}^{\infty}\frac{(-\mu_j^{\beta_2} t^{\alpha_2})^m}{\Gamma(\alpha_2 m+1)}\frac{\Gamma(\beta_2 m+1)}{\Gamma(\beta_2 m+1-y_j)} \nonumber \\
 &=\sum_{\mathcal{A}(k,n)}\prod_{j=1}^{k}\frac{(-1)^{x_j+y_j}}{x_j!y_j!}{}_2\psi_2\left[\begin{array}{cc}(1,1)&(1,\beta_1) \\(1,\alpha_1)&(1-x_j,\beta_1)\end{array}\bigg|-\lambda_j^{\beta_1} t^{\alpha_1}\right]{}_2\psi_2\left[\begin{array}{cc}(1,1)&(1,\beta_2) \\(1,\alpha_2)&(1-x_j,\beta_2)\end{array}\bigg|-\mu_j^{\beta_2} t^{\alpha_2}\right].\nonumber 
\end{align}
\end{remark}
Next, we obtain the p.m.f. of GSTFSP-II (see Definition~\ref{definition_gstfsp2}).
\begin{theorem}\label{p.m.f..gstfsp2}
    The p.m.f. of GSTFSP-II is given by
    \begin{equation*}
        P(\mathcal{S}_{\beta}^{\alpha}(t)=n)=\sum_{\mathcal{A}(k,n)}\prod_{j=1}^{k}\frac{\lambda_j^{x_j}\mu_j^{y_j}}{x_j!y_j!}\mathbb{E}(e^{-D_\beta(Y_\alpha(t))(\Lambda+T)}D_\beta(Y_\alpha(t))^{x_j+y_j})\quad\forall~n\in \mathbb{Z},
    \end{equation*}
    where $\displaystyle\mathcal{A}(k,n)= \{(x_1,\dots,x_k,y_1,\dots,y_k)\mid x_j\ge 0,\,y_j\ge 0,\,\sum_{j=1}^{k}j(x_j-y_j)=n\}$.
\end{theorem}
\begin{proof}
    Consider a GSP which has a weighted sum representation (see Proposition \ref{weighted.proposition}) given by $\{\mathcal{S}(t)=\sum_{j=1}^{k}jS_j(t)\}_{t\ge 0}$, where $\{S_j(t)\}_{t\ge 0}$ for $j=1,\dots,k$ are $k$ independent Skellam processes with parameters $\lambda_1,\dots,\lambda_k$ and $\mu_1,\dots,\mu_k$.  Then the p.m.f. of a GSTFSP-II can be obtained as follows:
\begin{align*}
    P(\mathcal{S}_{\beta}^{\alpha}(t)=n)=P(\mathcal{S}(D_\beta(Y_\alpha(t)))=n)&=\int_{0}^{\infty}P(\mathcal{S}(D_\beta(Y_\alpha(t)))=n|D_\beta(Y_\alpha(t))\in dz)P(D_\beta(Y_\alpha(t))\in dz)\\
    &=\int_{0}^{\infty}P(\mathcal{S}(z)=n)P(D_\beta(Y_\alpha(t))\in dz)\\
    &=\sum_{\mathcal{A}(k,n)}\int_{0}^{\infty}e^{-z(\Lambda+T)}\prod_{j=1}^{k}\frac{\lambda_j^{x_j}\mu_j^{y_j}}{x_j!y_j!}z^{x_j+y_j}P(D_\beta(Y_\alpha(t))\in dz)\\
    &=\sum_{\mathcal{A}(k,n)}\prod_{j=1}^{k}\frac{\lambda_j^{x_j}\mu_j^{y_j}}{x_j!y_j!}\int_{0}^{\infty}e^{-z(\Lambda+T)}z^{x_j+y_j}P(D_\beta(Y_\alpha(t))\in dz)\\
    &=\sum_{\mathcal{A}(k,n)}\prod_{j=1}^{k}\frac{\lambda_j^{x_j}\mu_j^{y_j}}{x_j!y_j!}\mathbb{E}(e^{-D_\beta(Y_\alpha(t))(\Lambda+T)}(D_\beta(Y_\alpha(t)))^{x_j+y_j}).
\end{align*}
\end{proof}

\subsection{Correct probability mass functions of some special cases}

First note that the sum of independent Skellam processes is a Skellam process but their weighted sum, where the weights are all possible jump sizes of the corresponding counting processes, cannot be a Skellam process. So the p.m.f.s of some Skellam processes introduced recently in the literature, which have weighted sum representations in terms of other Skellam processes (see Proposition \ref{weighted.proposition}), are incorrect. For example, a Skellam process of order $k$ is the weighted sum and not the simple sum of $k$ independent Skellam processes as implied by the incorrect p.m.f. in (38) of \citet{Gupta2020} (see \citet{Cinque2025}). As another example, note that the p.m.f. of a GSP (see (3.3) of \citet{Kataria2022b}) is not identifiable (and hence incorrect) since two distinct parameter sets for the corresponding GCPs can have the same sum. We provide the correct p.m.f.s of these and related processes as follows.

 Let $\{N_{1j}(t)\}_{t\ge 0}$ and $\{N_{2j}(t)\}_{t\ge 0}$ for $1\le j\le k$ be independent Poisson processes with rates $\lambda_j$ and $\mu_j$, respectively. Moreover let $\{M_1(t)\}_{t\ge 0}$ and $\{M_2(t)\}_{t\ge 0}$ be two independent GCPs. The correct p.m.f. \eqref{pmf_GSP} of a generalized Skellam process (GSP) can be alternatively obtained as follows.
\begin{align*}
P(\mathcal{S}(t)=n)=P(M_{1}(t)-M_{2}(t)=n)&=P(\sum_{j=1}^{k}jN_{1j}(t)-\sum_{j=1}^{k}jN_{2}(t)=n)\\
&=\sum_{\substack{x_j\ge 0,y_j\ge 0\\\sum_{j=1}^{k}j(x_j-y_j)=n}}\prod_{j=1}^{k}P(N_{1j}(t)=x_j,~N_{2j}(t)=y_j)\\
&=\sum_{\mathcal{A}(k,n)}\prod_{j=1}^{k}(e^{-\lambda_j t}\frac{(\lambda_j t)^{x_j}}{x_j!})(e^{-\mu_j t}\frac{(\mu_j t)^{y_j}}{y_j!})\\
&=\sum_{\mathcal{A}(k,n)}e^{-t(\Lambda+T)}\prod_{j=1}^{k}\frac{\lambda_j^{x_j}\mu_j^{y_j}}{x_j!y_j!}t^{x_j+y_j},
\end{align*}
where $\Lambda=\sum_{j=1}^{k}\lambda_j$, $T=\sum_{j=1}^{k}\mu_j$ and $\displaystyle\mathcal{A}(k,n)= \{(x_1,\dots,x_k,y_1,\dots,y_k)\mid x_j\ge 0,~y_j\ge 0,~\sum_{j=1}^{k}j(x_j-y_j)=n\}$.

Similarly, the correct p.m.f of a Skellam process of order $k$ (SPoK) is given by
    \begin{equation*}
        P(\mathcal{S}^k(t)=n)=\sum_{\mathcal{A}(k,n)}e^{-kt(\lambda+\mu)}\prod_{j=1}^{k}\frac{\lambda^{x_j}\mu^{y_j}}{x_j!y_j!}t^{x_j+y_j}.
    \end{equation*}

Next, we provide below the correct p.m.f.s of the fractional Skellam process of order $k$ or FSPoK-II (see \citet{kataria2024}) and the generalized fractional Skellam process or GFSP-II (see \citet{Kataria2022b}) obtained from Theorem \ref{p.m.f..gstfsp2}.
\begin{align*}
P(S^{k,\alpha}(t)=n)&=\sum_{\mathcal{A}(k,n)}\prod_{j=1}^{k}\frac{\lambda^{x_j}\mu^{y_j}}{x_j!y_j!}\mathbb{E}(e^{-kY_\alpha(t)(\lambda+\mu)}(Y_\alpha(t))^{x_j+y_j}), \\
P(\mathcal{S}^{\alpha}(t)=n)&=\sum_{\mathcal{A}(k,n)}\prod_{j=1}^{k}\frac{\lambda_j^{x_j}\mu_j^{y_j}}{x_j!y_j!}\mathbb{E}(e^{-Y_\alpha(t)(\Lambda+T)}(Y_\alpha(t))^{x_j+y_j}). 
\end{align*}
The closed-form p.m.f.s of FSP-I, FSPoK-I and GFSP-I obtained from Theorem \ref{p.m.f..gstfsp1} are
\begin{align*}
  P(S^{\alpha_1,\alpha_2}(t)=n)&=\sum_{\mathcal{A}(1,n)}\frac{\lambda^{x}\mu^{y}}{x!y!}((-\partial_\lambda)^{x}E_{\alpha_1,1}(-\lambda t^{\alpha_1}))((-\partial_\mu)^{y}E_{\alpha_2,1}(-\mu t^{\alpha_2})),\\
  P(S^{k,\alpha_1,\alpha_2}(t)=n)&=\sum_{\mathcal{A}(k,n)}\prod_{j=1}^{k}\frac{\lambda^{x_j}\mu^{y_j}}{x_j!y_j!}((-\partial_{\lambda})^{x_j}E_{\alpha_1,1}(-\lambda t^{\alpha_1}))((-\partial_{\mu})^{y_j}E_{\alpha_2,1}(-\mu t^{\alpha_2})),\\
  P(\mathcal{S}^{\alpha_1,\alpha_2}(t)=n)&=\sum_{\mathcal{A}(k,n)}\prod_{j=1}^{k}\frac{\lambda_j^{x_j}\mu_j^{y_j}}{x_j!y_j!}((-\partial_{\lambda_j})^{x_j}E_{\alpha_1,1}(-\lambda_j t^{\alpha_1}))((-\partial_{\mu_j})^{y_j}E_{\alpha_2,1}(-\mu_j t^{\alpha_2})),
\end{align*}
respectively. Finally the correct p.m.f. of the non-homogeneous generalized Skellam process (NGSP) (see \citet{Tathe2024}) denoted by $\{\mathcalboondox{S}(t)=\mathcal{M}_1(t)-\mathcal{M}_2(t)\}_{t\ge 0}$, where $\{\mathcal{M}_1(t)\}_{t\ge 0}$ and $\{\mathcal{M}_2(t)\}_{t\ge 0}$ are independent non-homogeneous generalized counting processes (NGCP) with rate functions $\lambda_1(t),\ldots,\lambda_k(t)$ and $\mu_1(t),\ldots,\mu_k(t)$ respectively, is given by
    \begin{equation*}
        P(\mathcalboondox{S}(t)=n)=\sum_{\mathcal{A}(k,n)}\prod_{j=1}^{k}\frac{(\Lambda_j(t))^{x_j}(T_j(t))^{y_j}}{x_j!y_j!}e^{-\sum_{j=1}^{k}(\Lambda_j(t)+T_j(t))},
    \end{equation*}
    where $\Lambda_j(t)=\int_{0}^{t}\lambda_j(u)du$ and $T_j(t)=\int_{0}^{t}\mu_j(u)du$. It follows that the correct p.m.f. of the non-homogeneous generalized fractional Skellam process (NGFSP-II) denoted by $\{\mathcalboondox{S}^\alpha(t)=\mathcalboondox{S}(Y_\alpha(t))\}_{t\ge 0}$ is given by 
    \begin{equation*}
        P(\mathcalboondox{S}^\alpha(t)=n)=\int_{0}^{\infty}P(\mathcalboondox{S}(u)=n)h_\alpha(u,t)du=\int_{0}^{\infty}(\sum_{\mathcal{A}(k,n)}\prod_{j=1}^{k}\frac{(\Lambda_j(u))^{x_j}(T_j(u))^{y_j}}{x_j!y_j!}e^{-\sum_{j=1}^{k}(\Lambda_j(u)+T_j(u))})h_\alpha(u,t)du.
    \end{equation*}

\subsection{Probability Generating Function of GSTFSP}
\begin{proposition}\label{p.g.f._GSTFSP1}
    The p.g.f. of GSTFSP-I for $|u|<1$ is given by
    \begin{equation*}
        \mathcal{G}_{\mathcal{S}^{\alpha_1,\alpha_2}_{\beta_1,\beta_2}}(u,t)=\sum_{i=0}^{\infty}\frac{(-(\sum_{j=1}^{k}(1-u^j)\lambda_j)^{\beta_1} t^{\alpha_1})^i}{\Gamma(i{\alpha_1}+1)}\times\sum_{l=0}^{\infty}\frac{(-(\sum_{j=1}^{k}(1-u^{-j})\mu_j)^{\beta_2} t^{\alpha_2})^l}{\Gamma(l{\alpha_2}+1)}.
    \end{equation*}
\end{proposition}
\begin{proof} Using the p.g.f. of GSTFCP (see \citet{Katariaarxiv}) and Lemma $3.1$ of \citet{Beghin2014}, the p.g.f. of GSTFSP-I can be obtained as follows:
\begin{align*}
    \mathcal{G}_{\mathcal{S}^{\alpha_1,\alpha_2}_{\beta_1,\beta_2}}(u,t)&=\mathbb{E}(u^{\mathcal{S}_{\beta_1,\beta_2}^{\alpha_1,\alpha_2}(t)})\\
    &=\mathbb{E}(\exp(-\sum_{j=1}^{k}(1-u^j)\lambda_jD_{\beta_1}(Y_{\alpha_1}(t))))\times\mathbb{E}(\exp(-\sum_{j=1}^{k}(1-u^{-j})\mu_jD_{\beta_2}(Y_{\alpha_2}(t)))) \\
    &=E_{\alpha_1,1}(-(\sum_{j=1}^{k}(1-u^j)\lambda_j)^{\beta_1} t^{\alpha_1})\times E_{\alpha_2,1}(-(\sum_{j=1}^{k}(1-u^{-j})\mu_j)^{\beta_2} t^{\alpha_2})\\
        &=\sum_{i=0}^{\infty}\frac{(-(\sum_{j=1}^{k}(1-u^j)\lambda_j)^{\beta_1} t^{\alpha_1})^i}{\Gamma(i{\alpha_1}+1)}\times\sum_{l=0}^{\infty}\frac{(-(\sum_{j=1}^{k}(1-u^{-j})\mu_j)^{\beta_2} t^{\alpha_2})^l}{\Gamma(l{\alpha_2}+1)}.
\end{align*}
\end{proof}
\begin{proposition}\label{p.g.f._GSTFSP2}
    The p.g.f. of GSTFSP-II for $|u|<1$ is given by
    \begin{equation*}
        \mathcal{G}_{\mathcal{S}^{\alpha}_{\beta}}(u,t)=\mathbb{E}(\exp(-D_\beta(Y_\alpha(t))\sum_{j=1}^{k}(\lambda_j(1-u^j)+\mu_j(1-u^{-j})))).
    \end{equation*}
\end{proposition}
\begin{proof}
    The p.g.f. of GSTFSP-II is given by
    \begin{align*}
        \mathcal{G}_{\mathcal{S}^{\alpha}_{\beta}}(u,t)=\mathbb{E}(u^{\mathcal{S}^{\alpha}_{\beta}(t)})&=\int_{0}^{\infty}\mathbb{E}(u^{\mathcal{S}(D_\beta(Y_\alpha(t)))}\big|D_\beta(Y_\alpha(t))\in dx)P(D_\beta(Y_\alpha(t))\in dx)\\
        &=\int_{0}^{\infty}\mathbb{E}(u^{\mathcal{S}(x)})P(D_\beta(Y_\alpha(t))\in dx)\\
        &=\int_{0}^{\infty}\exp(-x\sum_{j=1}^{k}(\lambda_j(1-u^j)+\mu_j(1-u^{-j})))P(D_\beta(Y_\alpha(t))\in dx)\\
        &=\mathbb{E}(\exp(-D_\beta(Y_\alpha(t))\sum_{j=1}^{k}(\lambda_j(1-u^j)+\mu_j(1-u^{-j})))).
    \end{align*}
\end{proof}
The following result connects GSTFSP-I with TFPPs using uniform random variables. This representation of the probability generating function of a process in terms of the minimum of uniform random variables follows the approach introduced in (1.12) of \citet{Orsingher2012}.

\begin{proposition}
Let $X_i$ and $Y_i$ for $i\ge 1$ be i.i.d. uniform random variables in $[0,1]$. Further let $\{N^{\alpha_1}(t,\Lambda^{\beta_1})\}_{t\ge 0}$ and $\{N^{\alpha_2}(t,T^{\beta_2})\}_{t\ge 0}$ be TFPPs with intensities $\Lambda^{\beta_1}$ and $T^{\beta_2}$ respectively such that $\min_{0\le i\le N^{\alpha_1}(t,\Lambda^{\beta_1})}X^{1/\beta_1}_{i}=\min_{0\le i\le N^{\alpha_2}(t,T^{\beta_2})}Y^{1/\beta_2}_{i}=1$ when $N^{\alpha_1}(t,\Lambda^{\beta_1})=N^{\alpha_2}(t,T^{\beta_2})=0$. Then for $0<u<1$, we have
\begin{equation*}
\mathcal{G}_{\mathcal{S}^{\alpha_1,\alpha_2}_{\beta_1,\beta_2}}(u,t)=P(\min_{0\le i\le N^{\alpha_1}(t,\Lambda^{\beta_1})}X^{1/\beta_1}_{i}\ge 1-\frac{1}{\Lambda}\sum_{j=1}^{k}\lambda_j u^{j})P(\min_{0\le i\le N^{\alpha_2}(t,T^{\beta_2})}Y^{1/\beta_2}_{i}\ge 1-\frac{1}{T}\sum_{j=1}^{k}\mu_j u^{-j}).
\end{equation*}
\end{proposition}
\begin{proof}
From \citet{Katariaarxiv}, the p.g.f. of $\{M_{1\beta}^{\alpha}(t)\}_{t\ge 0}$ for $0<u<1$ can be written as 
\begin{equation*}
\mathbb{E}(u^{M_{1\beta_1}^{\alpha_1}(t)}) = P(\min_{0\le i\le N^{\alpha_1}(t,\Lambda^{\beta_1})}X^{1/\beta_1}_{i}\ge 1-\frac{1}{\Lambda}\sum_{j=1}^{k}\lambda_j u^{j}).
\end{equation*}
Similarly the p.g.f. of $\{M_{2\beta}^{\alpha}(t)\}_{t\ge 0}$ for $0<u<1$ can be written as
\begin{equation*}
\mathbb{E}(u^{M_{2\beta_2}^{\alpha_2}(t)}) = P(\min_{0\le i\le N^{\alpha_2}(t,T^{\beta_2})}Y^{1/\beta_2}_{i}\ge 1-\frac{1}{T}\sum_{j=1}^{k}\mu_j u^{j}).
\end{equation*}
Hence the p.g.f. of GSTFSP-I is given by 
\begin{align*}
\mathcal{G}_{\mathcal{S}^{\alpha}_{\beta}}(u,t)&=\mathbb{E}(u^{M_{1\beta_1}^{\alpha_1}(t)})\times\mathbb{E}(u^{-M_{2\beta_2}^{\alpha_2}(t)})\nonumber\\
&=P(\min_{0\le i\le N^{\alpha_1}(t,\Lambda^{\beta_1})}X^{1/\beta_1}_{i}\ge 1-\frac{1}{\Lambda}\sum_{j=1}^{k}\lambda_j u^{j})P(\min_{0\le i\le N^{\alpha_2}(t,T^{\beta_2})}Y^{1/\beta_2}_{i}\ge 1-\frac{1}{T}\sum_{j=1}^{k}\mu_j u^{-j}).
\end{align*}
\end{proof}
Next, we obtain the moments of GSTFCP, GSTFSP-I and GSTFSP-II. We begin by deriving the mean, variance and covariance of GSTFCP.
\subsection{Mean, variance and covariance for GSTFCP}
Note that $\{M(t)\}_{t\geq 0}$ is a homogeneous L\'{e}vy process with $M(0)=0$ and $\{D_\beta(Y_\alpha(t))\}_{t\geq 0}$ is a non-decreasing process independent of $\{M(t)\}_{t\geq 0}$. Hence, using Theorem 2.1 of \citet{Leonenko2014}, the mean, variance and covariance of the GSTFCP $M_{\beta}^{\alpha}(t)=M(D_\beta(Y_\alpha(t)))$ can be obtained.
The mean, variance and covariance for various special cases can be obtained from those of GSTFCP by making suitable substitutions and noting that $Y_1(t)=t$ and $D_1(t)=t$ (see \citet{Maheshwari2019}). However, integer order moments of GSTFCP, STFPP, GSFCP, SFPP, STFPPoK and SFPPoK do not exist as integer order moments for the $\beta$-stable subordinator $\{D_\beta(t)\}_{t\geq 0}$ do not exist (see section 2.1 of \citet{KatariaTCSTFPP2022}). Hence, the correlation structures of these processes and consequently their long and short range dependence properties (see \citet{Maheshwari2019}) cannot be determined.
\subsection{Fractional order moments for GSTFCP and GSFCP}
Since integer moments do not exist for GSTFCP and GSFCP, we may consider their fractional moments. The $q^{th}$ order fractional moments (see \citet{Kumar2019}) of a positive random variable $X$ having Laplace transform $\tilde{f}(u)$ for $q\in(n-1,n)$ are given by
\begin{equation}
    \mathbb{E}(X^q)=\frac{(-1)^n}{\Gamma(n-q)}\int_{0}^{\infty}\frac{d^n}{du^n}[\tilde{f}(u)]u^{n-q-1}du.\label{fractional_moment}
\end{equation}
Putting $n=1$ in \eqref{fractional_moment}, the $q^{th}$ order fractional moments of GSTFCP and GSFCP where $q\in(0,1)$ can be obtained as follows:
\begin{align*}
    \mathbb{E}((M_{\beta}^{\alpha}(t))^q)&=\frac{(-1)}{\Gamma(1-q)}\int_{0}^{\infty}\frac{d}{du}E_{\alpha,1}(-(\sum_{j=1}^{k}(1-e^{-uj})\lambda_j)^\beta t^\alpha)u^{-q}du\\
    &=\frac{(-1)}{\Gamma(1-q)}\int_{0}^{\infty}\frac{d}{du}(\sum_{i=0}^{\infty}\frac{(-(\sum_{j=1}^{k}(1-e^{-uj})\lambda_j)^\beta t^\alpha)^i}{\Gamma(i\alpha+1)})u^{-q}du\\
    &=\frac{(-1)}{\Gamma(1-q)}\sum_{i=0}^{\infty}\frac{(-t^\alpha)^i(\beta i)}{\Gamma(i\alpha+1)}\sum_{j=1}^{k}j\lambda_j\int_{0}^{\infty}(\sum_{l=1}^{k}(1-e^{-ul})\lambda_l)^{\beta i-1}e^{-uj}u^{-q}du\\
\text{and}\quad\mathbb{E}((M_{\beta}(t))^q)&=\frac{(-1)}{\Gamma(1-q)}\int_{0}^{\infty}\frac{d}{du}\exp(-t(\sum_{j=1}^{k}(1-e^{-uj})\lambda_j)^\beta )u^{-q}du\\
&=\frac{\beta t}{\Gamma(1-q)}\sum_{j=1}^{k}j\lambda_j\int_{0}^{\infty}(\sum_{l=1}^{k}(1-e^{-ul})\lambda_l)^{\beta-1}\exp(-t(\sum_{l=1}^{k}(1-e^{-ul})\lambda_l)^\beta - uj)u^{-q}du.
\end{align*}
As the integrands are bounded, the fractional moments for GSTFCP and GSFCP exist. However, numerical integration techniques are required to evaluate the integrals since they are difficult to compute analytically. 

\begin{remark}
Putting $\beta=1$ in the fractional moments of GSTFCP and GSFCP, we can obtain the fractional moments of GFCP and GCP respectively. Similarly, the fractional moments of PPoK and the Poisson process can be obtained from those of GCP by substituting $\lambda_1=\lambda_2=\cdots=\lambda_k=\lambda$ and $k=1$ (in addition) respectively. 
\end{remark}

 \subsection{Mean, variance and covariance for GSTFSP}\label{moments_GSTFSP}
 The mean, variance and covariance of GSTFSP-I and GSTFSP-II can be obtained from those of GSTFCP.
Similarly, the mean, variance and covariance for various special cases can be obtained from those of GSTFSP-I and GSTFSP-II. However, integer order moments of GSTFSP-I, GSTFSP-II, STFSP-I, STFSP-II, GSFSP-I, GSFSP-II, SFSP-I, SFSP-II, STFSPoK-I, STFSPoK-II, SFSPoK-I and SFSPoK-II don't exist as integer order moments for the $\beta$-stable subordinator $\{D_\beta(t)\}_{t\geq 0}$ don't exist. Hence the LRD and SRD properties of these processes can't be established (see \citet{Maheshwari2016}).
\begin{remark}
 The moments for all time fractional processes including GFCP, GCP, FPoK and their Skellam versions GFSP-I, GFSP-II, FSP-I, FSP-II, FSPoK-I and FSPoK-II exist due to the existence of moments of the inverse $\alpha$-stable subordinator $\{Y_\alpha(t)\}_{t\geq 0}$ (see \citet{Aletti2018}).   
\end{remark}

\section{Recurrence relations}\label{sec 4}

In this section, we focus on recurrence relations for these processes. Since the p.g.f. of GSTFSP-I involves product of Mittag-Leffler functions (see Proposition \ref{p.g.f._GSTFSP1}), the recurrence relation satisfied by its p.m.f. can't be obtained explicitly using its p.g.f. So we obtain these relations for some special cases. The next result provides a recursive characterization of the state probabilities of GSFSP-I, where the probability of the current state is represented as the difference between linear combinations of probabilities of past and future states. The coefficients depend on the jump sizes, the arrival rates and the fractional parameters of the corresponding GSFCPs, capturing the influence of both upward and downward jumps. 
\begin{theorem}\label{recurrence.GSFSP}
    The state probabilities of GSFSP-I satisfy the following recurrence relation for $|u|<1$, $n\geq 1$:
\begin{equation*}
    p(n,t)=\frac{t}{n}(\beta_1(\sum_{j=1}^{k}(1-u^j)\lambda_j)^{\beta_1-1}\sum_{j=1}^{k}j\lambda_jp(n-j,t)-\beta_2(\sum_{j=1}^{k}(1-u^{-j})\mu_j)^{\beta_2-1}\sum_{j=1}^{k}j\mu_jp(n+j,t)).
\end{equation*}
\end{theorem}

\begin{proof}
Using the definition of p.g.f. for GSFSP-I, we have
\begin{align}
\frac{d}{du} \mathcal{G}_{\mathcal{S}_{\beta_1,\beta_2}}(u,t)&=\frac{d}{du}\sum_{i=0}^{\infty}u^{i}p(i,t)=\sum_{i=0}^{\infty}(i+1)p(i+1,t)u^{i}.\label{dif_p.g.f.1}
\end{align}
From Proposition \ref{p.g.f._GSTFSP1}, for $\alpha_1=\alpha_2=1$, we have 
\begin{align}
\mathcal{G}_{\mathcal{S}_{\beta_1,\beta_2}}(u,t)
    &=\exp\{-t(\sum_{j=1}^{k}(1-u^j)\lambda_j)^{\beta_1}-t(\sum_{j=1}^{k}(1-u^{-j})\mu_j)^{\beta_2}\}.\label{p.g.f._GSFSP}
\end{align}
Differentiating \eqref{p.g.f._GSFSP}, we get
\begin{align}
    &\frac{d}{du}\mathcal{G}_{\mathcal{S}_{\beta_1,\beta_2}}(u,t)\nonumber\\
    &=\sum_{i=0}^{\infty}u^ip(i,t)(-t)[\beta_1(\sum_{j=1}^{k}(1-u^j)\lambda_j)^{\beta_1-1}(\sum_{j=1}^{k}-ju^{j-1}\lambda_j)+\beta_2(\sum_{j=1}^{k}(1-u^{-j})\mu_j)^{\beta_2-1}(\sum_{j=1}^{k}ju^{-j-1}\mu_j)].\label{dif_p.g.f.2}
\end{align}
Comparing the RHS of \eqref{dif_p.g.f.1} and \eqref{dif_p.g.f.2}, we get
\begin{align}
    &\sum_{i=0}^{\infty}(i+1)p(i+1,t)u^{i}\nonumber\\
    &=\sum_{i=0}^{\infty}u^ip(i,t)(-t)[\beta_1(\sum_{j=1}^{k}(1-u^j)\lambda_j)^{\beta_1-1}(\sum_{j=1}^{k}-ju^{j-1}\lambda_j)+\beta_2(\sum_{j=1}^{k}(1-u^{-j})\mu_j)^{\beta_2-1}(\sum_{j=1}^{k}ju^{-j-1}\mu_j)]\nonumber\\
    &=(\sum_{j=1}^{k}(1-u^j)\lambda_j)^{\beta_1-1}\sum_{j=1}^{k}t\beta_1\sum_{i=j-1}^{\infty}j\lambda_jp(i-j+1,t)u^i-(\sum_{j=1}^{k}(1-u^{-j})\mu_j)^{\beta_2-1}\sum_{j=1}^{k}t\beta_2\sum_{i=-j-1}^{\infty}j\mu_jp(i+j+1,t)u^i\nonumber\\
    &=(\sum_{j=1}^{k}(1-u^j)\lambda_j)^{\beta_1-1}\sum_{j=1}^{k}j\lambda_jt\beta_1[\sum_{i=j-1}^{k-1}p(i-j+1,t)u^i+\sum_{i=k}^{\infty}p(i-j+1,t)u^i]\nonumber\\
    &-(\sum_{j=1}^{k}(1-u^{-j})\mu_j)^{\beta_2-1}\sum_{j=1}^{k}j\mu_jt\beta_2[\sum_{i=-j-1}^{k-1}p(i+j+1,t)u^i+\sum_{i=k}^{\infty}p(i+j+1,t)u^i]\nonumber\\
    &=(\sum_{j=1}^{k}(1-u^j)\lambda_j)^{\beta_1-1}\sum_{i=0}^{k-1}\sum_{j=1}^{i+1}j\lambda_jt\beta_1 p(i-j+1,t)u^i+(\sum_{j=1}^{k}(1-u^j)\lambda_j)^{\beta_1-1}\sum_{i=k}^{\infty}\sum_{j=1}^{k}j\lambda_jt\beta_1 p(i-j+1,t)u^i\nonumber\\
    &-(\sum_{j=1}^{k}(1-u^{-j})\mu_j)^{\beta_2-1}\sum_{i=-k-1}^{k-1}\sum_{j=-i-1}^{k}j\mu_jt\beta_2 p(i+j+1,t)u^i-(\sum_{j=1}^{k}(1-u^{-j})\mu_j)^{\beta_2-1}\sum_{i=k}^{\infty}\sum_{j=1}^{k}j\mu_jt\beta_2 p(i+j+1,t)u^i
    \,\,.\label{comparison1}
\end{align}
Now equating the coefficients of $u^{i}$ for $i\leq k-1$ on both sides of \eqref{comparison1}, we obtain
\begin{equation*}
    (i+1)p(i+1,t)=(\sum_{j=1}^{k}(1-u^j)\lambda_j)^{\beta_1-1}\sum_{j=1}^{i+1}j\lambda_jt\beta_1 p(i-j+1,t)-(\sum_{j=1}^{k}(1-u^{-j})\mu_j)^{\beta_2-1}\sum_{j=-i-1}^{k}j\mu_jt\beta_2 p(i+j+1,t)
\end{equation*}
for $-k\leq n\leq k$ which reduces to 
\begin{equation}
    p(n,t)=\frac{t}{n}(\beta_1(\sum_{j=1}^{k}(1-u^j)\lambda_j)^{\beta_1-1}\sum_{j=1}^{i+1}j\lambda_j p(n-j,t)-\beta_2(\sum_{j=1}^{k}(1-u^{-j})\mu_j)^{\beta_2-1}\sum_{j=-i-1}^{k}j\mu_j p(n+j,t))\label{comparison2}
\end{equation}
as in the first summation $p(n-(n+1),t), p(n-(n+2),t),\dots,p(n-k,t)$ are all vanishing and in the second summation $\mu_{-n}, \mu_{-n+1},\dots,\mu_{-1}$ all are zeroes. Also $p(-i,t)=0$ for $1\leq i\leq k$. 

Again, by equating the coefficients of $u^{i}$ for $i\geq k$ on both sides of \eqref{comparison1}, we obtain
\begin{equation*}
    (i+1)p(i+1,t)=(\sum_{j=1}^{k}(1-u^j)\lambda_j)^{\beta_1-1}\sum_{j=1}^{k}j\lambda_jt\beta_1 p(i-j+1,t)-(\sum_{j=1}^{k}(1-u^{-j})\mu_j)^{\beta_2-1}\sum_{j=1}^{k}j\mu_jt\beta_2 p(i+j+1,t)
\end{equation*}
for $n\geq k+1$ which reduces to
\begin{equation}
    p(n,t)=\frac{t}{n}(\beta_1(\sum_{j=1}^{k}(1-u^j)\lambda_j)^{\beta_1-1}\sum_{j=1}^{k}j\lambda_j p(n-j,t)-\beta_2(\sum_{j=1}^{k}(1-u^{-j})\mu_j)^{\beta_2-1}\sum_{j=1}^{k}j\mu_j p(n+j,t)).\label{comparison3}
\end{equation}
Combining \eqref{comparison2} and \eqref{comparison3}, we have for $n\geq 1$
\begin{equation*}
     p(n,t)=\frac{t}{n}(\beta_1(\sum_{j=1}^{k}(1-u^j)\lambda_j)^{\beta_1-1}\sum_{j=1}^{k}j\lambda_jp(n-j,t)-\beta_2(\sum_{j=1}^{k}(1-u^{-j})\mu_j)^{\beta_2-1}\sum_{j=1}^{k}j\mu_jp(n+j,t)).
\end{equation*}
\end{proof}

\begin{remark}
    The recurrence relation satisfied by the state probabilities of SFSP-I can be obtained by putting $k=1$ in Theorem \ref{recurrence.GSFSP} as follows:
    \begin{equation*}
     p(n,t)=\frac{t}{n}(\beta_1(1-u)^{\beta_1-1}\lambda^{\beta_1} p(n-1,t)-\beta_2(1-u^{-1})^{\beta_2-1}\mu^{\beta_2} p(n+1,t)).
    \end{equation*}
\end{remark}

\begin{remark}
    For $\beta_1=\beta_2=1$, Theorem \ref{recurrence.GSFSP} reduces to the recurrence relation satisfied by the state probabilities of GSP (see recurrence relation for NGSP of \citet{Tathe2024} for non-homogeneous rates)
    \begin{equation*}
        p(n,t)=\frac{t}{n}\sum_{j=1}^{k}j(\lambda_jp(n-j,t)-\mu_jp(n+j,t)).
    \end{equation*}
    Similarly, using the  integral representation of the p.g.f. of GFSP in terms of GSP, the recurrence relation for GFSP can be obtained as follows:
    \begin{equation*}
        p^\alpha(n,t)=\frac{t}{n}\sum_{j=1}^{k}j(\lambda_jp^\alpha(n-j,t)-\mu_jp^\alpha(n+j,t)),
    \end{equation*}
    where $p^\alpha(n,t)$ is the p.m.f. of GFSP.
\end{remark}
The next result follows from Theorem \ref{recurrence.GSFSP}.
\begin{proposition}\label{recurrence.GSFCP}
    The state probabilities $q(n,t)$ of GSFCP satisfy the following recurrence relation
\begin{equation*}
    q(n,t)=\frac{t\beta}{n}(\sum_{j=1}^{k}(1-u^j)\lambda_j)^{\beta-1}\sum_{j=1}^{min\{n,k\}}j\lambda_jq(n-j,t),\quad n\geq 1.
\end{equation*}
\end{proposition}
\begin{remark}
    For $\beta=1$, Proposition \ref{recurrence.GSFCP} reduces to the recurrence relation satisfied by the state probabilities of GCP (see \citet{KatariaGFCP}). Using the  integral representation of the p.g.f. of GFCP in terms of GCP, the recurrence relation for GFCP can be obtained as follows:
    \begin{equation*}
        q^\alpha(n,t)=\frac{t}{n}\sum_{j=1}^{k}j\lambda_jq^\alpha(n-j,t),
    \end{equation*}
     where $q^\alpha(n,t)$ is the p.m.f. of GFCP.
\end{remark}

\section{Transition Probabilities, Arrival and First Passage Times}\label{sec 5}
In this section, we begin by deriving the infinitesimal transition probabilities of GSTFSP-I. Unlike the classical birth-death process, GSTFSP-I allows transitions using integer-valued jumps of sizes $1,2,\ldots,k$ in both upward and downward directions with transition rates $\lambda_i$ and $\mu_i$, respectively. We subsequently derive the distributions of the arrival and first-passage times for GSTFSP-I and GSTFSP-II.
\begin{theorem}\label{transition_gstfsp}
    The transition probabilities of GSTFSP-I are given by
\begin{equation*}
    P(\mathcal{S}_{\beta_1,\beta_2}^{\alpha_1,\alpha_2}(t+\delta)=m\mid\mathcal{S}_{\beta_1,\beta_2}^{\alpha_1,\alpha_2}(t)=n)=\begin{cases}
        \lambda_i\delta+o(\delta), & m>n,\, m=n+i, i=1,2,...k;\\
        \mu_i\delta+o(\delta), & m<n, m=n-i,\, i=1,2,...k;\\
        1-(\Lambda+T)\delta+o(\delta), & m=n;\\
        o(\delta), & \text{otherwise},
    \end{cases}
\end{equation*}
i.e., at most $k$ events can occur in a very small interval of time and even though the probability for more than $k$ events is non-zero, it is negligible.
\end{theorem}
\begin{proof}
Denote $\{M_{1\beta_1}^{\alpha_1}(t)\}_{t\ge 0}$ and $\{M_{2\beta_2}^{\alpha_2}(t)\}_{t\ge 0}$ to be the first and second processes respectively in the definition of GSTFSP. Then      
\begin{align*}
    P(\mathcal{S}_{\beta_1,\beta_2}^{\alpha_1,\alpha_2}(t+\delta)=n+i\mid\mathcal{S}_{\beta_1,\beta_2}^{\alpha_1,\alpha_2}(t)=n)
    &=\sum_{j=1}^{k-i} P(\text{the first process has $i+j$ arrivals and the second process has $j$ arrivals}) \\
    &+ P(\text{the first process has $i$ arrivals and the second process has $0$ arrivals}) + o(\delta)\\
    &= \sum_{j=1}^{k-i} (\lambda_{i+j}\delta+o(\delta))*(\mu_{j}\delta+o(\delta)) + (\lambda_i\delta+o(\delta))*(1-T\delta+o(\delta)) \\
    &+ o(\delta) =  \lambda_i\delta+o(\delta)\,\,.
\end{align*}
Similarly, we have 
\begin{align*}
P(\mathcal{S}_{\beta_1,\beta_2}^{\alpha_1,\alpha_2}(t+\delta)=n-i\mid\mathcal{S}_{\beta_1,\beta_2}^{\alpha_1,\alpha_2}(t)=n)
    &=\sum_{j=1}^{k-i} P(\text{the first process has $j$ arrivals and the second process has $i+j$ arrivals}) \\
    &+ P(\text{the first process has $0$ arrivals and the second process has $i$ arrivals}) + o(\delta)\\
    &= \sum_{j=1}^{k-i} (\lambda_j\delta+o(\delta))*(\mu_{i+j}\delta+o(\delta)) + (1-\Lambda\delta+o(\delta))*(\mu_i\delta+o(\delta)) \\
    &+ o(\delta)=\mu_i\delta+o(\delta)\,\,.
\end{align*}
Finally
\begin{align*}
    P(\mathcal{S}_{\beta_1,\beta_2}^{\alpha_1,\alpha_2}(t+\delta)=n\mid\mathcal{S}_{\beta_1,\beta_2}^{\alpha_1,\alpha_2}(t)=n)
    &=\sum_{j=1}^{k} P(\text{the first process has $j$ arrivals and the second process has $j$ arrivals}) \\
    &+ P(\text{the first process has $0$ arrivals and the second process has $0$ arrivals}) + o(\delta)\\
    &= \sum_{j=1}^{k} (\lambda_j\delta+o(\delta))*(\mu_j\delta+o(\delta)) + (1-\Lambda\delta+o(\delta))*(1-T\delta+o(\delta)) \\
    &+ o(\delta) = 1-\Lambda\delta -T\delta+o(\delta)
\end{align*}
which proves the theorem.
\end{proof}
\begin{remark}
    Since the transition probabilities of the GSTFSP-I do not depend explicitly on $\alpha_1$, $\alpha_2$, $\beta_1$ and $\beta_2$, the corresponding transition probabilities for its special cases can be derived directly. In particular, the transition probabilities of GSP, GFSP-I, and GSFSP-I coincide with those given in Theorem \ref{transition_gstfsp}. For STFSPoK-I, FSPoK-I and SFSPoK-I, they are obtained by substituting  $\lambda_i=\lambda, \mu_i=\mu$ for $i=1,2,\dots,k$. Since GSTFSP-II is constructed by time-changing GSP with $D_\beta(Y_\alpha(t))$, its transition probabilities can be derived using the integral representation of the p.m.f. of GSTFSP-II in terms of GSP. So these transition probabilities are the same as those stated in Theorem \ref{transition_gstfsp}. The transition probabilities of GFSP-II, GSFSP-II, STFSPoK-II, FSPoK-II, and SFSPoK-II follow in the same manner as discussed for the Type I case.
\end{remark}

\begin{proposition}\label{arrival.gstfsp}
Consider the $n^{th} $ arrival time of GSTFSP-I defined as 
\begin{equation*}
\tau_n=min\{t\geq 0: \mathcal{S}_{\beta_1,\beta_2}^{\alpha_1,\alpha_2}(t)=n\}\,.
\end{equation*}
The distribution function of $\tau_n$ is given by 
\begin{align*}
    F^{\mathcal{S}_{\beta_1,\beta_2}^{\alpha_1,\alpha_2}}_{\tau_n}(t)&=\sum_{i\in(0,t]}(\sum_{\mathcal{A}(k,n)}\prod_{j=1}^{k}\frac{\lambda_j^{x_j}\mu_j^{y_j}}{x_j!y_j!}((-\partial_{\lambda_j})^{x_j}E_{\alpha_1,1}(-\lambda_j^{\beta_1} i^{\alpha_1}))((-\partial_{\mu_j})^{y_j}E_{\alpha_2,1}(-\mu_j^{\beta_2} i^{\alpha_2}))).
\end{align*}  
\end{proposition}
\begin{proof}
For $n\in \mathbb{Z}$, we have
\begin{align*}
F^{\mathcal{S}_{\beta_1,\beta_2}^{\alpha_1,\alpha_2}}_{\tau_n}(t)=1-P(\tau_n > t)=1-P(\bigcap_{i\in (0,t]}\{\mathcal{S}_{\beta_1,\beta_2}^{\alpha_1,\alpha_2}(i)\neq n\})&=1-P(\bigcap_{i\in (0,t]}\{\mathcal{S}_{\beta_1,\beta_2}^{\alpha_1,\alpha_2}(i)= n\}^c)\\
        &=P(\bigcup_{i\in (0,t]}\{\mathcal{S}_{\beta_1,\beta_2}^{\alpha_1,\alpha_2}(i)= n\})\\
        &= \sum_{i\in(0,t]}P(\mathcal{S}_{\beta_1,\beta_2}^{\alpha_1,\alpha_2}(i)= n).
\end{align*}
The result follows from Theorem~\ref{p.m.f..gstfsp1}. 
\end{proof}
 
\begin{proposition}\label{first passage.ngsp}
    Let $T_n$ be the time of the first upcrossing of the level $n$ for GSTFSP-I given by
\begin{equation*}
    T_n=\text{inf}\{s\geq 0:\mathcal{S}_{\beta_1,\beta_2}^{\alpha_1,\alpha_2}(s)\geq n\}\,.
\end{equation*}
The distribution function of $T_n$ is given by 
\begin{align*}
    F^{\mathcal{S}_{\beta_1,\beta_2}^{\alpha_1,\alpha_2}}_{T_n}(t)&=\sum_{i\in (0,t]}\sum_{m=n}^{\infty}(\sum_{\mathcal{A}(k,m)}\prod_{j=1}^{k}\frac{\lambda_j^{x_j}\mu_j^{y_j}}{x_j!y_j!}((-\partial_{\lambda_j})^{x_j}E_{\alpha,1}(-\lambda_j^\beta i^\alpha))((-\partial_{\mu_j})^{y_j}E_{\alpha,1}(-\mu_j^\beta i^\alpha))).
\end{align*}
\end{proposition}

\begin{proof} For $n\in \mathbb{Z}$, we have
    \begin{align*}
    F^{\mathcal{S}_{\beta_1,\beta_2}^{\alpha_1,\alpha_2}}_{T_n}(t)=1-P(T_n> t)=1-P(\bigcap_{i\in (0,t]}\{\mathcal{S}_{\beta_1,\beta_2}^{\alpha_1,\alpha_2}(i)<n\})
    &=P(\bigcap_{i\in (0,t]}\{\mathcal{S}_{\beta_1,\beta_2}^{\alpha_1,\alpha_2}(i)\le n-1\})^c\\
    &=P(\bigcup_{i\in (0,t]}\{\mathcal{S}_{\beta_1,\beta_2}^{\alpha_1,\alpha_2}(i) > n-1\})\\
    &= \sum_{i\in (0,t]}\sum_{m=n}^{\infty}P(\mathcal{S}_{\beta_1,\beta_2}^{\alpha_1,\alpha_2}(i)= m).
\end{align*}
The result follows from Theorem~\ref{p.m.f..gstfsp1}.
\end{proof}
 Similarly, the distribution functions of $n^{th}$ arrival and first passage time of GSTFSP-II are given by
 \begin{align*}
        F^{\mathcal{S}_{\beta}^{\alpha}}_{\tau_n}(t)&=\sum_{i\in(0,t]}(\sum_{\mathcal{A}(k,n)}\prod_{j=1}^{k}\frac{\lambda_j^{x_j}\mu_j^{y_j}}{x_j!y_j!}\mathbb{E}(e^{-D_\beta(Y_\alpha(i))(\Lambda+T)}D_\beta(Y_\alpha(i))^{x_j+y_j}))\\
         F^{\mathcal{S}_{\beta}^{\alpha}}_{T_n}(t)&=\sum_{i\in (0,t]}\sum_{m=n}^{\infty}(\sum_{\mathcal{A}(k,m)}\prod_{j=1}^{k}\frac{\lambda_j^{x_j}\mu_j^{y_j}}{x_j!y_j!}\mathbb{E}(e^{-D_\beta(Y_\alpha(i))(\Lambda+T)}D_\beta(Y_\alpha(i))^{x_j+y_j})).
 \end{align*}
\begin{remark}
    The arrival time and the first passage time distributions of the special cases STFSP-I, STFSP-II, GSFSP-I, GSFSP-II, SFSP-I, SFSP-II, STFSPoK-I, STFSPoK-II, SFSPoK-I, SFSPoK-II, FSP-I, FSP-II, GFSP-I, GFSP-II, FSPoK-I and FSPoK-II can be obtained using their respective p.m.f.s.
\end{remark}

\section{Increment process}\label{sec 6}
In this section, we introduce the increment process along with its fractional version of GSFSP-I, GSFSP-II and study their marginals.

\begin{definition}
 The increment process $\{I_{\mathcal{S}_{\beta_1,\beta_2}}(t,v)\}_{t\ge 0}$ of a GSFSP-I is a stochastic process defined as
\begin{align*}
I_{\mathcal{S}_{\beta_1,\beta_2}}(t,v)&=\mathcal{S}_{\beta_1,\beta_2}(t+v)-\mathcal{S}_{\beta_1,\beta_2}(v),\quad v\ge 0.
\end{align*}
\end{definition}
\begin{definition}
 The increment process $\{I_{\mathcal{S}_{\beta}}(t,v)\}_{t\ge 0}$ of a GSFSP-II is a stochastic process defined as
\begin{align*}
I_{\mathcal{S}_{\beta}}(t,v)&=\mathcal{S}_{\beta}(t+v)-\mathcal{S}_{\beta}(v),\quad v\ge 0.
\end{align*}
\end{definition}

\begin{proposition}\label{p.m.f._increment_GSFSP-I}
    The marginals of the increment process of a GSFSP-I are given by
    \begin{align*}
    P(I_{\mathcal{S}_{\beta_1,\beta_2}}(t,v)=n)&=\sum_{\mathcal{A}(k,n)}\prod_{j=1}^{k}\frac{\lambda_j^{x_j}\mu_j^{y_j}}{x_j!y_j!}((-\partial_{\lambda_j})^{x_j}e^{-\lambda_j^{\beta_1} t})((-\partial_{\mu_j})^{y_j}e^{-\mu_j^{\beta_2} t}),\quad n\in\mathbb{Z}.
\end{align*}
\end{proposition}
\begin{proof}
Let $\mathcal{S}_{\beta_1,\beta_2}(t)=M_{1\beta_1}(t)-M_{2\beta_2}(t)$ where $\{M_{1\beta_1}(t)\}_{t\geq 0}$ and $\{M_{2\beta_2}(t)\}_{t\geq 0}$ are two independent GSFCPs. The increment process of $\{\mathcal{S}_{\beta_1,\beta_2}(t)\}_{t\ge 0}$ is
\begin{align*}
I_{\mathcal{S}_{\beta_1,\beta_2}}(t,v)&=(M_{1\beta_1}(t+v)-M_{2\beta_2}(t+v))-(M_{1\beta_1}(v)-M_{2\beta_2}(v)) \\
&=
(M_{1\beta_1}(t+v)-M_{1\beta_1}(v))-(M_{2\beta_2}(t+v)-M_{2\beta_2}(v))\\
&\overset{d}{=} M_{1\beta_1}(t)-M_{2\beta_2}(t)=\mathcal{S}_{\beta_1,\beta_2}(t)
\end{align*}
since GSFCPs have stationary increments. Hence 
\begin{align} 
    P(I_{\mathcal{S}_{\beta_1,\beta_2}}(t,v)=n)
    &=P(\mathcal{S}_{\beta_1,\beta_2}(t)=n)=\sum_{\mathcal{A}(k,n)}\prod_{j=1}^{k}\frac{\lambda_j^{x_j}\mu_j^{y_j}}{x_j!y_j!}((-\partial_{\lambda_j})^{x_j}e^{-\lambda_j^{\beta_1} t)})((-\partial_{\mu_j})^{y_j}e^{-\mu_j^{\beta_2} t}).\label{incrementGSTFSP1}
 \end{align}
\end{proof}

 \begin{remark}
Putting $v=0$ in Proposition \ref{p.m.f._increment_GSFSP-I}, we get the p.m.f. of GSFSP-I.
 \end{remark}
\begin{proposition}\label{p.m.f._increment_GSFSP-II}
    The marginals of the increment process of a GSFSP-II are given by
    \begin{align*}
    P(I_{\mathcal{S}_{\beta}}(t,v)=n)&=\sum_{\mathcal{A}(k,n)}\prod_{j=1}^{k}\frac{\lambda_j^{x_j}\mu_j^{y_j}}{x_j!y_j!}\mathbb{E}(e^{-D_\beta(t)(\Lambda+T)}D_\beta(t)^{x_j+y_j}),\quad n\in\mathbb{Z}.
\end{align*}
\end{proposition}
\begin{proof}
Let $\mathcal{S}_{\beta}(t)=\mathcal{S}(D_\beta(t))$, where $\{\mathcal{S}(t)\}_{t\ge 0}$ is a GSP and $\{D_\beta(t)\}_{t\ge 0}$ is an independent $\beta$-stable subordinator. The increment process of $\{\mathcal{S}_{\beta}(t)\}_{t\ge 0}$ is
\begin{align*}
I_{\mathcal{S}_{\beta}}(t,v)&=\mathcal{S}_\beta(t+v)-\mathcal{S}_\beta(v)\overset{d}{=}\mathcal{S}_{\beta}(t)
\end{align*}
since GSFSP-II has stationary increments. Hence 
\begin{align} 
    P(I_{\mathcal{S}_{\beta}}(t,v)=n)
    &=P(\mathcal{S}_{\beta}(t)=n)=\sum_{\mathcal{A}(k,n)}\prod_{j=1}^{k}\frac{\lambda_j^{x_j}\mu_j^{y_j}}{x_j!y_j!}\mathbb{E}(e^{-D_\beta(t)(\Lambda+T)}D_\beta(t)^{x_j+y_j}).\label{incrementGSTFSP-II}
 \end{align}
\end{proof}
\begin{remark}
Putting $v=0$ in Proposition \ref{p.m.f._increment_GSFSP-II}, we get the p.m.f. of GSFSP-II.
 \end{remark}
 
\begin{definition}
 The fractional increment process $\{I_{\mathcal{S}_{\beta_1,\beta_2}}^\alpha(t,v)\}_{t\geq 0}$ of GSFSP-I is a stochastic process defined as
\begin{align*}
I_{\mathcal{S}_{\beta_1,\beta_2}}^\alpha(t,v)&=\mathcal{S}_{\beta_1,\beta_2}(Y_\alpha(t)+v)-\mathcal{S}_{\beta_1,\beta_2}(v),
\end{align*}
where $\{Y_\alpha(t)\}_{t\geq 0}$ is an inverse $\alpha$-stable subordinator independent of the GSFSP-I.
\end{definition}
\begin{definition}\label{definition_fractional_increment_GSFSP-II}
 The fractional increment process $\{I_{\mathcal{S}_{\beta}}^\alpha(t,v)\}_{t\geq 0}$ of GSFSP-II is a stochastic process defined as
\begin{align*}
I_{\mathcal{S}_{\beta}}^\alpha(t,v)&=\mathcal{S}_{\beta}(Y_\alpha(t)+v)-\mathcal{S}_{\beta}(v),
\end{align*}
where $\{Y_\alpha(t)\}_{t\geq 0}$ is an inverse $\alpha$-stable subordinator independent of the GSFSP-I.
\end{definition}

\begin{proposition}\label{p.m.f._fractional_increment_GSFSP-I}
The marginals of the fractional increment process of a GSFSP-I are given by
\begin{align*}
P(I_{\mathcal{S}_{\beta_1,\beta_2}}^\alpha(t,v)=n)&=\int_{0}^{\infty}\sum_{\mathcal{A}(k,n)}\prod_{j=1}^{k}\frac{\lambda_j^{x_j}\mu_j^{y_j}}{x_j!y_j!}((-\partial_{\lambda_j})^{x_j}e^{-\lambda_j^{\beta_1} u})((-\partial_{\mu_j})^{y_j}e^{-\mu_j^{\beta_2} u})h_\alpha(u,t) du,  \quad \forall~n\in \mathbb{Z}\,
\end{align*}
where $h_\alpha(u,t)$ is the density of $\{Y_\alpha(t)\}_{t\geq 0}$. 
\end{proposition}

\begin{proof}
For $n\in\mathbb{Z}$, we have 
\begin{equation*}
 P(I_{\mathcal{S}_{\beta_1,\beta_2}}^\alpha(t,v)=n)=P(I_{\mathcal{S}_{\beta_1,\beta_2}}(Y_\alpha(t),v)=n)=\int_{0}^{\infty}P(I_{\mathcal{S}_{\beta_1,\beta_2}}(u,v)=n)h_\alpha(u,t)du.  
\end{equation*}
The result now follows using Proposition \ref{p.m.f._increment_GSFSP-I}. 
\end{proof}
\begin{proposition}\label{p.m.f._fractional_increment_GSFSP-II}
The marginals of the fractional increment process of a GSFSP-II are given by
\begin{align*}
P(I_{\mathcal{S}_{\beta}}^\alpha(t,v)=n)&=\int_{0}^{\infty}\sum_{\mathcal{A}(k,n)}\prod_{j=1}^{k}\frac{\lambda_j^{x_j}\mu_j^{y_j}}{x_j!y_j!}\mathbb{E}(e^{-D_\beta(t)(\Lambda+T)}D_\beta(t)^{x_j+y_j})h_\alpha(u,t) du,  \quad \forall~n\in \mathbb{Z}\,
\end{align*}
where $h_\alpha(u,t)$ is the density of $\{Y_\alpha(t)\}_{t\geq 0}$. 
\end{proposition}

\begin{proof}
The proof is similar to that of Proposition \ref{p.m.f._fractional_increment_GSFSP-I} and
the result follows using Proposition \ref{p.m.f._increment_GSFSP-II}. 
\end{proof}

\begin{theorem}
    The characteristic function of the fractional increment process of a GSFSP-I satisfies the following system of fractional differential-integral equations: 
    \begin{equation*}
        \frac{d^\alpha}{dt^\alpha}\phi_{I_{\mathcal{S}_{\beta_1,\beta_2}}^{\alpha}(t,v)}(\xi)=\int_{0}^{\infty}\{(\sum_{j=1}^{k}(1-e^{i\xi j})\lambda_j)^{\beta_1} + (\sum_{j=1}^{k}(1-e^{-i\xi j})\mu_j)^{\beta_2}\}\phi_{I_{\mathcal{S}_{\beta}}(u,v)}(\xi)h_\alpha(u,t)du.
    \end{equation*}
\end{theorem}
\begin{proof}
The characteristic function of $I_{\mathcal{S}_{\beta_1,\beta_2}}^\alpha(t,v)$ is given by
\begin{equation}\label{incre1}
 \phi_{I_{\mathcal{S}_{\beta_1,\beta_2}}^{\alpha}(t,v)}(\xi)=\int_{0}^{\infty}\phi_{I_{\mathcal{S}_{\beta_1,\beta_2}}(u,v)}(\xi)h_\alpha(u,t)du
\end{equation}
where $\phi_{I_{\mathcal{S}_{\beta_1,\beta_2}}(u,v)}$ is the characteristic function of $I_{\mathcal{S}_{\beta}}(t,v)$. By substituting $\alpha_1=\alpha_2=1$ in Proposition \ref{p.g.f._GSTFSP1}, we obtain
\begin{equation}\label{incre2}
\phi_{I_{\mathcal{S}_{\beta_1,\beta_2}}(u,t)}= \exp[-t\{(\sum_{j=1}^{k}(1-e^{i\xi j})\lambda_j)^{\beta_1} + (\sum_{j=1}^{k}(1-e^{-i\xi j})\mu_j)^{\beta_2}\}].   
\end{equation}

Now taking the Laplace transform of \eqref{incre1} and using \eqref{incre2} along with the Laplace transform of the inverse $\alpha$-stable subordinator (see section \ref{stable subordinator}), we get

\begin{align}
     \mathcal{L}\{\phi_{I_{\mathcal{S}_{\beta_1,\beta_2}}^{\alpha}(t,v)}(\xi)\}(s)&=\int_{0}^{\infty}\phi_{I_{\mathcal{S}_{\beta_1,\beta_2}}(u,v)}(\xi)\tilde{h}_\alpha(u,s)du\nonumber\\
     &=s^{\alpha-1}\int_{0}^{\infty}\exp[-u\{(\sum_{j=1}^{k}(1-e^{i\xi j})\lambda_j)^{\beta_1} + (\sum_{j=1}^{k}(1-e^{-i\xi j})\mu_j)^{\beta_2}\}] e^{-us^{\alpha}}du \nonumber\\
        &=s^{\alpha-1}\{[\frac{e^{-us^{\alpha}}}{s^{\alpha}}\exp(-u\{(\sum_{j=1}^{k}(1-e^{i\xi j})\lambda_j)^{\beta_1} + (\sum_{j=1}^{k}(1-e^{-i\xi j})\mu_j)^{\beta_2}\})]_{0}^{\infty}\nonumber\\ 
        &+\frac{1}{s^\alpha}\int_{0}^{\infty}\{(\sum_{j=1}^{k}(1-e^{i\xi j})\lambda_j)^{\beta_1} + (\sum_{j=1}^{k}(1-e^{-i\xi j})\mu_j)^{\beta_2}\}\nonumber\\ 
        & \times \exp(-u\{(\sum_{j=1}^{k}(1-e^{i\xi j})\lambda_j)^{\beta_1} + (\sum_{j=1}^{k}(1-e^{-i\xi j})\mu_j)^{\beta_2}\})\times e^{-us^{\alpha}}du
        \},\nonumber
\end{align}
which implies
\begin{align}
s^{\alpha}\mathcal{L}\{\phi_{I_{\mathcal{S}_{\beta_1,\beta_2}}^{\alpha}(t,v)}(\xi)\}(s)-s^{\alpha-1}&=s^{\alpha-1}\int_{0}^{\infty}\{(\sum_{j=1}^{k}(1-e^{i\xi j})\lambda_j)^{\beta_1} + (\sum_{j=1}^{k}(1-e^{-i\xi j})\mu_j)^{\beta_2}\} \nonumber\\ 
        & \times \exp(-u\{(\sum_{j=1}^{k}(1-e^{i\xi j})\lambda_j)^{\beta_1} + (\sum_{j=1}^{k}(1-e^{-i\xi j})\mu_j)^{\beta_2}\})\times e^{-us^{\alpha}}du. \nonumber
\end{align}
    Hence
    \begin{align}
    \mathcal{L}\{\frac{d^\alpha}{dt^\alpha}\phi_{I_{\mathcal{S}_{\beta_1,\beta_2}}^{\alpha}(t,v)}(\xi)\}(s) 
       &=\int_{0}^{\infty}\{(\sum_{j=1}^{k}(1-e^{i\xi j})\lambda_j)^{\beta_1} + (\sum_{j=1}^{k}(1-e^{-i\xi j})\mu_j)^{\beta_2}\}\phi_{I_{\mathcal{S}_{\beta_1,\beta_2}}(u,v)}(\xi)\tilde{h}_\alpha(u,t)du\nonumber
\end{align}
where we used the Laplace transform of fractional Caputo–Djrbashian derivative (see section \ref{caputo}). Next, on taking the inverse Laplace transform, we get
\begin{align}
    \frac{d^\alpha}{dt^\alpha}\phi_{I_{\mathcal{S}_{\beta_1,\beta_2}}^{\alpha}(t,v)}(\xi)&=\int_{0}^{\infty}\{(\sum_{j=1}^{k}(1-e^{i\xi j})\lambda_j)^{\beta_1} + (\sum_{j=1}^{k}(1-e^{-i\xi j})\mu_j)^{\beta_2}\}\phi_{I_{\mathcal{S}_{\beta_1,\beta_2}}(u,v)}(\xi)h_\alpha(u,t)du\nonumber.
\end{align}
\end{proof}
 \begin{remark}
     Note that the increment process of GSTFSP-II is defined as
     \begin{equation*}
     I_{\mathcal{S}_{\beta}^{\alpha}}(t,v)=\mathcal{S}_{\beta}^{\alpha}(t+v)-\mathcal{S}_{\beta}^{\alpha}(v)=\mathcal{S}_\beta(Y_{\alpha}(t+v))-\mathcal{S}_\beta(Y_{\alpha}(v)),\quad v\ge 0,
     \end{equation*}
     which is different from Definition \ref{p.m.f._fractional_increment_GSFSP-II}.
     The marginals of $I_{\mathcal{S}_{\beta}^{\alpha}}(t,v)$ can be obtained using the joint distribution of $\mathcal{S}_{\beta}^{\alpha}(t+v)$ and $\mathcal{S}_{\beta}^{\alpha}(v)$ which are not independent. The increment process of GSTFSP-I can be defined in a similar manner. Moreover the increments of GSFSP-I and GSFSP-II are independent whereas the increments of GSTFSP-I and GSTFSP-II are not.  
 
 \end{remark}   
\section{Asymptotic behaviour of tail probability}\label{sec 7}
In this section, we explore the asymptotic behaviour of the tail probability for GSTFSP along with some special cases. For this study, we use the Tauberian theorem (see  Example (c) on p. 447 of \citet{WilliamF2008}) as stated below. 
\begin{theorem}\label{Tauberian_theorem}
For a probability distribution $F$ with characteristic function $\phi$, as $u\rightarrow 0$ and $x\rightarrow\infty$, we have
\begin{equation}
1-\phi(u)\sim u^{1-\rho}L(\frac{1}{u}) \quad \text{if and only if}\quad 1-F(x)\sim\frac{1}{\Gamma(\rho)}x^{\rho-1}L(x)
\end{equation}
for $\rho>0$, where $L$ is a slowly varying function satisfying $\frac{L(tx)}{L(t)}\to 1$ as $t\to \infty$ for any fixed $x$.    
\end{theorem}
The following result shows the asymptotic behaviour of the tail probability of GSTFCP.
\begin{theorem}\label{Asymptotic_GSTFCP}
    The tail probability of GSTFCP has the following asymptotic behavior:
    \begin{equation*}
        P(M_{\beta}^{\alpha}(t)>x)\sim\frac{x^{-\beta} (\sum_{j=1}^{k}j\lambda_j)^{\beta}t^{\alpha}}{\Gamma(1-\beta)\Gamma(\alpha+1)} \quad\text{as}\,\, x\to \infty.
    \end{equation*}
\end{theorem}
\begin{proof}
Let $F$ denote the c.d.f. of GSTFCP and $\psi(u)$ denote its Laplace transform for $u>0$. Using integration by parts, we have
\begin{equation*}
\int_{0}^{\infty} e^{-ux}F(x) dx = \frac{\psi(u)}{u} \implies \int_{0}^{\infty} e^{-ux}(1-F(x)) dx = \frac{1-\psi(u)}{u}. 
\end{equation*}
Hence the Laplace transform of the tail probability of GSTFCP is given by
    \begin{align*}
        \int_{0}^{\infty}e^{-u x}P(M_{\beta}^{\alpha}(t)>x)dx
        &=\frac{1-E_{\alpha,1}(-(\sum_{j=1}^{k}(1-e^{-u j})\lambda_j)^\beta t^\alpha)}{u}.
    \end{align*}
    Since $1-e^{-u}\sim u$ as $u\to 0$, we may neglect second and higher order terms in the Taylor series expansion of the Mittag Leffler function around $0$. Thus we have
    \begin{align}
        \frac{1-E_{\alpha,1}(-(\sum_{j=1}^{k}(1-e^{-uj})\lambda_j)^\beta t^\alpha)}{u}
        &\sim\frac{1-(1-\frac{(\sum_{j=1}^{k}(1-e^{-uj})\lambda_j)^\beta t^\alpha}{\Gamma(\alpha+1)})}{u}\nonumber\\
        &=\frac{(\sum_{j=1}^{k}(1-e^{-uj})\lambda_j)^\beta t^\alpha}{u\Gamma(\alpha+1)}\nonumber\\
        &\sim\frac{(\sum_{j=1}^{k}(uj)\lambda_j)^\beta t^\alpha}{u\Gamma(\alpha+1)}=\frac{u^\beta(\sum_{j=1}^{k}j\lambda_j)^\beta t^\alpha}{u\Gamma(\alpha+1)}\label{GSTFCP_tail}.
    \end{align}

    Comparing \eqref{GSTFCP_tail} with the first asymptotic relation in Theorem \ref{Tauberian_theorem}, we obtain $\rho=1-\beta$ and $L(1/u)=(\sum_{j=1}^{k}j\lambda_j)^\beta t^\alpha/\Gamma(\alpha+1)$. It follows that
    \begin{equation*}
        P(M_{\beta}^{\alpha}(t)>x)=1-F(x)\sim\frac{x^{-\beta} (\sum_{j=1}^{k}j\lambda_j)^{\beta}t^{\alpha}}{\Gamma(1-\beta)\Gamma(\alpha+1)}\quad \text{for fixed}\,\, t > 0\,\, \text{as}\,\, x\to \infty.
    \end{equation*}
\end{proof}
Theorem \ref{Asymptotic_GSTFCP} shows that the probability of observing large values of a GSTFCP for fixed $t$ decreases polynomially according to $x^{-\beta}$, with the proportionality constant depending on the weighted jump intensity $\sum_{j=1}^{k}j\lambda_j$ and the fractional parameters $\alpha$ and $\beta$. The next result provides the asymptotic behaviour of the tail probability of GSFCP. 
\begin{theorem}\label{Asymptotic_GSFCP}
The tail probability of GSFCP has the following asymptotic behaviour:
\begin{equation*}
P(M_{\beta}(t)>x)\sim\frac{x^{-\beta} (\sum_{j=1}^{k}j\lambda_j)^{\beta}t}{\Gamma(1-\beta)} ,\quad\text{as}\,\, x\to \infty.
\end{equation*}
\end{theorem}

\begin{proof}
The proof is similar to that of Theorem \ref{Asymptotic_GSTFCP} and hence omitted.
\end{proof}
Using Theorem \ref{Asymptotic_GSTFCP}, the tail probability of GSTFSP-I is given by
    \begin{align}
        P(\mathcal{S}_{\beta_1,\beta_2}^{\alpha_1,\alpha_2}(t)>x)&=P(M_{1\beta_1}^{\alpha_1}(t)-M_{2\beta_2}^{\alpha_2}(t)>x)\nonumber\\
        &=\sum_{y=0}^{\infty}P(M_{1\beta_1}^{\alpha_1}(t)>x+y\big|M_{2\beta_2}^{\alpha_2}(t)=y)P(M_{2\beta_2}^{\alpha_2}(t)=y)\nonumber\\
        &=\sum_{y=0}^{\infty}P(M_{1\beta_1}^{\alpha_!}(t)>x+y)P(M_{2\beta_2}^{\alpha_2}(t)=y)
        \nonumber\\ 
        &\sim\sum_{y=0}^{\infty}\frac{(x+y)^{-\beta_1} (\sum_{j=1}^{k}j\lambda_j)^{\beta_1}t^{\alpha_1}}{\Gamma(1-\beta)\Gamma(\alpha_1+1)}\times(\sum_{\Omega(k,y)}\prod_{j=1}^{k}\frac{\mu_j^{x_j}}{x_j!}(-\partial_{\mu_j})^{x_j}E_{\alpha_2,1}(-\mu_j^{\beta_2} t^{\alpha_2}))\nonumber\\
        &=\frac{ (\sum_{j=1}^{k}j\lambda_j)^{\beta_1}t^{\alpha_1}}{\Gamma(1-\beta_1)\Gamma(\alpha_1+1)}\sum_{y=0}^{\infty}(x+y)^{-\beta_1}\sum_{\Omega(k,y)}\prod_{j=1}^{k}\frac{\mu_j^{x_j}}{x_j!}(-\partial_{\mu_j})^{x_j}E_{\alpha_2,1}(-\mu_j^{\beta_2} t^{\alpha_2})\quad\text{as}\,\, x\to\infty.\nonumber
    \end{align}
Since the series $\sum_{y=0}^{\infty}(x+y)^{-\beta_1}$ diverges, closed-form expressions for the asymptotic tail probabilities of GSTFSP-I and GSFSP-I can not be obtained using the Tauberian theorem. 

\section{Limiting Results and Infinite Divisbility}\label{sec 8}
In this section, we obtain the asymptotic distribution, law of iterated logarithms and infinite divisibility of GSTFSP, GSTFCP and related processes. We start with the limiting results for these processes. 
\begin{theorem}
  For $\frac{\alpha_1}{\beta_1}=\frac{\alpha_2}{\beta_2}=c$, the GSTFSP-I satisfies the following limiting result :
\begin{equation*}
\lim_{t\rightarrow\infty}\frac{\mathcal{S}_{\beta_1,\beta_2}^{\alpha_1,\alpha_2}(t)}{t^{c}} \overset{d}{=}(Y_{\alpha_1}(1))^{1/\beta_1}D_{\beta_1}(1)\sum_{j=1}^{k}j\lambda_j-(Y_{\alpha_2}(1))^{1/\beta_2}D_{\beta_2}(1)\sum_{j=1}^{k}j\mu_j.
\end{equation*}
\end{theorem}
\begin{proof}
Let $\{M_1(t)\}_{t\geq 0}$ and $\{M_2(t)\}_{t\geq 0}$ be two independent GCPs with rates $\lambda_1,\lambda_2,\dots,\lambda_k$ and $\mu_1,\mu_2,\dots,\mu_k$ respectively.
By the Strong Law of Large Numbers for subordinators (see \citet{Bertoin2004}, p. 92), we have

\begin{equation}
    \lim_{t\rightarrow \infty}\frac{M_1(t)}{t}=\sum_{j=1}^{k}j\lambda_j\quad\text{a.s.},\qquad \lim_{t\rightarrow \infty}\frac{M_2(t)}{t}=\sum_{j=1}^{k}j\mu_j\quad\text{a.s.} \label{GCP/t_limit}
\end{equation}
Now using the self-similarity properties of $\{D_\beta(t)\}_{t\geq 0}$ and $\{Y_\alpha(t)\}_{t \geq 0}$ in the definition of GCP, we have
\begin{align} M_{1\beta_1}^{\alpha_1}(t)&\overset{d}{=}M_1(D_{\beta_1}(Y_{\alpha_1}(t)))\overset{d}{=}M_1(D_{\beta_1}(t^{\alpha_1} Y_{\alpha_1}(1)))\overset{d}{=}M_1(t^{\alpha_1/\beta_1}(Y_{\alpha_1}(1))^{1/\beta_1}D_{\beta_1}(1))\nonumber\\
\text{and}~M_{2\beta_2}^{\alpha_2}(t)&\overset{d}{=}M_2(D_{\beta_2}(Y_{\alpha_2}(t)))\overset{d}{=}M_2(D_{\beta_2}(t^{\alpha_2} Y_{\alpha_2}(1)))\overset{d}{=}M_2(t^{\alpha_2/\beta_2}(Y_{\alpha_2}(1))^{1/\beta_2}D_{\beta_2}(1)).\nonumber
\end{align}
Therefore
\begin{align}
    \lim_{t\rightarrow \infty}(\frac{M_{1\beta_1}^{\alpha_1}(t)}{t^{\alpha_1/\beta_1}}-\frac{M_{2\beta_2}^{\alpha_2}(t)}{t^{\alpha_2/\beta_2}})&\overset{d}{=}\lim_{t\rightarrow \infty}\frac{M_1(t^{\alpha_1/\beta_1}(Y_{\alpha_1}(1))^{1/\beta_1}D_{\beta_1}(1))}{t^{\alpha_1/{\beta_1}}}-\lim_{t\rightarrow \infty}\frac{M_2(t^{\alpha_2/\beta_2}(Y_{\alpha_2}(1))^{1/\beta_2}D_{\beta_2}(1))}{t^{\alpha_2/{\beta_2}}}\nonumber\\
    &\overset{d}{=}(Y_{\alpha_1}(1))^{1/\beta_1}D_{\beta_1}(1)\lim_{t\rightarrow \infty}\frac{M_1(t^{\alpha_1/\beta_1}(Y_{\alpha_1}(1))^{1/\beta_1}D_{\beta_1}(1))}{t^{\alpha_1/{\beta_1}}(Y_{\alpha_1}(1))^{1/\beta_1}D_{\beta_1}(1)}\nonumber\\
    &-(Y_{\alpha_2}(1))^{1/\beta_2}D_{\beta_2}(1)\lim_{t\rightarrow \infty}\frac{M_2(t^{\alpha_2/\beta_2}(Y_{\alpha_2}(1))^{1/\beta_2}D_{\beta_2}(1))}{t^{\alpha_2/{\beta_2}}(Y_{\alpha_2}(1))^{1/\beta_2}D_{\beta_2}(1)}\nonumber\\
    &\overset{d}{=}(Y_{\alpha_1}(1))^{1/\beta_1}D_{\beta_1}(1)\sum_{j=1}^{k}j\lambda_j-(Y_{\alpha_2}(1))^{1/\beta_2}D_{\beta_2}(1)\sum_{j=1}^{k}j\mu_j. \qquad(\text{using~ \eqref{GCP/t_limit}}) \nonumber
\end{align}
Hence, for $\frac{\alpha_1}{\beta_1}=\frac{\alpha_2}{\beta_2}=c$,
\begin{equation*}
\lim_{t\rightarrow\infty}\frac{\mathcal{S}_{\beta_1,\beta_2}^{\alpha_1,\alpha_2}(t)}{t^{c}} \overset{d}{=}(Y_{\alpha_1}(1))^{1/\beta_1}D_{\beta_1}(1)\sum_{j=1}^{k}j\lambda_j-(Y_{\alpha_2}(1))^{1/\beta_2}D_{\beta_2}(1)\sum_{j=1}^{k}j\mu_j.
\end{equation*}
\end{proof}

\begin{theorem}\label{GSTFSP.limit}
The GSTFSP-II satisfies the following limiting result :
\begin{equation*}
\lim_{t\rightarrow\infty}\frac{\mathcal{S}_{\beta}^{\alpha}(t)}{t^{\alpha/\beta}} \overset{d}{=} (Y_\alpha(1))^{1/\beta}D_\beta(1)\sum_{j=1}^{k}j(\lambda_j-\mu_j).
\end{equation*}
\end{theorem}

\begin{proof}
By Definition \eqref{definition.GSTFSP-II}, we have 
\begin{equation*}
    \mathcal{S}_{\beta}^{\alpha}(t)\overset{d}{=}\mathcal{S}(D_\beta(Y_\alpha(t))),
\end{equation*} 
where $\{\mathcal{S}(t)\}_{t\geq 0}$ is a GSP. Let $\mathcal{S}(t)=M_1(t)-M_2(t)$ where $\{M_1(t)\}_{t\geq 0}$ and $\{M_2(t)\}_{t\geq 0}$ are two independent GCPs with rates $\lambda_1,\lambda_2,\dots,\lambda_k$ and $\mu_1,\mu_2,\dots,\mu_k$ respectively.

 It follows from \eqref{GCP/t_limit} that
\begin{equation}
    \lim_{t\rightarrow \infty}\frac{\mathcal{S}(t)}{t}=\sum_{j=1}^{k}j(\lambda_j-\mu_j)\quad\text{a.s.}\implies\lim_{t\rightarrow \infty}\frac{\mathcal{S}(t)}{t}\overset{d}{=}\sum_{j=1}^{k}j(\lambda_j-\mu_j).\label{GSPlimit}
\end{equation}
Now using the self-similarity properties of $\{D_\beta(t)\}_{t\geq 0}$ and $\{Y_\alpha(t)\}_{t \geq 0}$ in the definition of GSTFSP, we have
\begin{equation}\label{GSTFSP_self-similarity} \mathcal{S}_{\beta}^\alpha(t)\overset{d}{=}\mathcal{S}(D_\beta(Y_\alpha(t)))\overset{d}{=}\mathcal{S}(D_\beta(t^\alpha Y_\alpha(1)))\overset{d}{=}\mathcal{S}(t^{\alpha/\beta}(Y_\alpha(1))^{1/\beta}D_\beta(1)).
\end{equation}
Therefore
\begin{align}
    \lim_{t\rightarrow \infty}\frac{\mathcal{S}_{\beta}^\alpha(t)}{t^{\alpha/\beta}}&\overset{d}{=}\lim_{t\rightarrow \infty}\frac{\mathcal{S}(t^{\alpha/\beta}(Y_\alpha(1))^{1/\beta}D_\beta(1))}{t^{\alpha/{\beta}}}\qquad(\text{using}~\eqref{GSTFSP_self-similarity})\nonumber\\
    &\overset{d}{=}(Y_\alpha(1))^{1/\beta}D_\beta(1)\lim_{t\rightarrow \infty}\frac{\mathcal{S}(t^{\alpha/\beta}(Y_\alpha(1))^{1/\beta}D_\beta(1))}{t^{\alpha/\beta}(Y_\alpha(1))^{1/\beta}D_\beta(1)}\nonumber\\
    &\overset{d}{=}(Y_\alpha(1))^{1/\beta}D_\beta(1)\sum_{j=1}^{k}j(\lambda_j-\mu_j). \qquad(\text{using~ \eqref{GSPlimit}}) \nonumber
\end{align}
\end{proof}

The next result follows from Theorem \ref{GSTFSP.limit}.
\begin{proposition}\label{limit.GSFSP-II.GFSP-II}
For $\alpha=1$, GSTFSP-II reduces to GSFSP-II $\{\mathcal{S}_{\beta}(t)=\mathcal{S}(D_\beta(t))\}_{t\geq 0}$ which satisfies 
\begin{equation*}
\lim_{t\rightarrow\infty}\frac{\mathcal{S}_{\beta}(t)}{t^{1/\beta}} \overset{d}{=} D_\beta(1)\sum_{j=1}^{k}j(\lambda_j-\mu_j).
\end{equation*}
Similarly for $\beta=1$, GSTFSP-II reduces to GFSP-II $\{\mathcal{S}^{\alpha}(t)=\mathcal{S}(Y_{\alpha}(t))\}_{t\ge 0}$ which satisfies (see \citet{Kataria2022b})
\begin{equation*}
\lim_{t\rightarrow\infty}\frac{\mathcal{S}^{\alpha}(t)}{t^{\alpha}} \overset{d}{=} Y_{\alpha}(1)\sum_{j=1}^{k}j(\lambda_j-\mu_j).
\end{equation*}
\end{proposition}

Now we prove a version of the law of iterated logarithm for GSTFSP. The following result will be used (see Theorem 14 of \cite{Bertoin2004}, p. 92).

\begin{lemma}\label{lil.Subordinator}
Let $\{D_f(t)\}_{t\ge 0}$ be a L\'{e}vy subordinator whose Laplace exponent $f$ is a regularly varying function at $0+$ with index $0<\gamma <1$, that is, $\lim_{x\rightarrow 0+}f(\lambda x)/f(x)=\lambda^{\gamma},~\lambda>0$. If
\begin{equation*}
g(t)=\frac{\log\log t}{\phi(t^{-1}\log\log t)}, \quad t>e
\end{equation*}
where $\phi$ is the inverse of $f$, then
\begin{equation*}
\lim\inf_{t\rightarrow\infty}\frac{D_f(t)}{g(t)} = \gamma(1-\gamma)^{(1-\gamma)/\gamma}~~~~~\text{a.s.} 
\end{equation*}
\end{lemma}
For the $\beta$-stable subordinator $\{D_{\beta}(t)\}_{t\ge 0}$, the Laplace transform is $\mathbb{E}(e^{-sD_{\beta}(t)})=e^{-ts^{\beta}}$ so the Laplace exponent is $f(s)=s^{\beta}$ with inverse $\phi(s)=s^{1/\beta}$. Note that $f$ is a regularly varying function with index $0<\beta\le 1$ since $\lim_{x\rightarrow 0+}f(\lambda s)/f(s)=\lambda^{\beta},~\lambda>0$. From Lemma \ref{lil.Subordinator}, if
\begin{equation*}
g(t)=\frac{\log\log t}{(t^{-1}\log\log t)^{1/\beta}}, \quad t>e   
\end{equation*}
then
\begin{equation*}
\lim\inf_{t\rightarrow\infty}\frac{D_\beta(t)}{g(t)} = \beta(1-\beta)^{(1-\beta)/\beta}~~~~~\text{a.s.} 
\end{equation*}
This leads to the next result.

\begin{theorem}\label{lil.GSTFSP-II}
The GSTFSP-II satisfies the following law of iterated logarithm:
\begin{equation*}
\lim\inf_{t\rightarrow\infty}\frac{\mathcal{S}_{\beta}^{\alpha}(t)}{g_{\alpha}(t)} \overset{d}{=} \sum_{j=1}^kj(\lambda_j-\mu_j)\beta(1-\beta)^{(1-\beta)/\beta}(Y_\alpha(1))^{1/\beta}     
\end{equation*}
where
\begin{equation*}
g_{\alpha}(t)=g(t^{\alpha})=\frac{\log\log t^{\alpha}}{(t^{-\alpha}\log\log t^{\alpha})^{1/\beta}},\quad t>e^{1/\alpha}.
\end{equation*}
\end{theorem}

\begin{proof}
Using the self-similarity properties of $\{D_{\beta}(t)\}_{t\ge 0}$ and $\{Y_{\alpha}(t)\}_{t\ge 0}$, we have 
\begin{equation}\label{GSTFSP-II.self-similarity} \mathcal{S}_{\beta}^\alpha(t)\overset{d}{=}\mathcal{S}(D_\beta(Y_\alpha(t)))\overset{d}{=}\mathcal{S}(D_\beta(t^\alpha Y_\alpha(1)))\overset{d}{=}\mathcal{S}((Y_\alpha(1))^{1/\beta}D_\beta(t^{\alpha})).
\end{equation}
It follows that
\begin{align*}
\lim\inf_{t\rightarrow\infty}\frac{\mathcal{S}_{\beta}^{\alpha}(t)}{g_{\alpha}(t)} &\overset{d}{=} \lim\inf_{t\rightarrow\infty}\frac{\mathcal{S}((Y_\alpha(1))^{1/\beta}D_\beta(t^{\alpha}))}{g_{\alpha}(t)} \nonumber \\
&=\lim\inf_{t\rightarrow \infty}\frac{\mathcal{S}((Y_\alpha(1))^{1/\beta}D_\beta(t^{\alpha}))}{(Y_\alpha(1))^{1/\beta}D_\beta(t^{\alpha})}\frac{(Y_\alpha(1))^{1/\beta}D_\beta(t^{\alpha})}{g_{\alpha}(t)}\nonumber\\
&=(Y_\alpha(1))^{1/\beta}\sum_{j=1}^kj(\lambda_j-\mu_j)\lim\inf_{t\rightarrow \infty}\frac{D_\beta(t^{\alpha})}{g_{\alpha}(t)} \nonumber \\
&=\sum_{j=1}^kj(\lambda_j-\mu_j)\beta(1-\beta)^{(1-\beta)/\beta}(Y_\alpha(1))^{1/\beta}.
\end{align*}
\end{proof}

\begin{corollary}\label{lil.GSFSP}
Substituting $\alpha=1$ in GSTFSP-II, we obtain the law of iterated logarithm for GSFSP-II using Theorem \ref{lil.GSTFSP-II} as follows:
\begin{equation*}
\lim\inf_{t\rightarrow\infty}\frac{\mathcal{S}_{\beta}(t)}{g(t)} \overset{d}{=} \sum_{j=1}^kj(\lambda_j-\mu_j)\beta(1-\beta)^{(1-\beta)/\beta}   
\end{equation*}
where
\begin{equation*}
g(t)=\frac{\log\log t}{(t^{-1}\log\log t)^{1/\beta}},\quad t>e.
\end{equation*}
\end{corollary}

\begin{remark}
The law of iterated logarithms for GSTFCP and GSFCP can be easily obtained from Theorem \ref{lil.GSTFSP-II} and Corollary \ref{lil.GSFSP} respectively by substituting $\lambda_j=0$ or $\mu_j=0$ for $j=1,\ldots,k$. 
\end{remark}
 
Next we provide some infinite divisibility results for GSTFSP-II, GSFCP and related processes. 
\begin{theorem}\label{id.GSTFSP}
The one-dimensional distributions of GSTFSP-II are not infinitely divisible.
\end{theorem}

\begin{proof}
Consider the GSTFSP-II $\{\mathcal{S}_{\beta}^{\alpha}(t)\}_{t\ge 0}$. From Theorem \ref{GSTFSP.limit}, we have
\begin{equation}\label{limit.GSTFSP}
\lim_{t\rightarrow\infty}\frac{\mathcal{S}_{\beta}^{\alpha}(t)}{t^{\alpha/\beta}} \overset{d}{=} (Y_\alpha(1))^{1/\beta}D_\beta(1)\sum_{j=1}^{k}j(\lambda_j-\mu_j).  
\end{equation}
Assume that $\mathcal{S}_{\beta}^{\alpha}(t)$ is infinitely divisible which implies that $\frac{\mathcal{S}_{\beta}^{\alpha}(t)}{t^{\alpha/\beta}}$ is infinitely divisible (see Proposition 2.1 of \citet{Steutel2004}). Since the limit of a sequence of infinitely divisible r.v.s is also infinitely divisible (see Proposition 2.2 of \citet{Steutel2004}), the RHS of \eqref{limit.GSTFSP} is infinitely divisible. Now $Y_\alpha(1)$ is not infinitely divisible (see \citet{Vellaisamy2018}) while $D_{\beta}(1)$, being a L\'{e}vy process, is infinitely divisible. Suppose $D_{\beta}(1)=X_1+\cdots + X_n$ where $X_1,\ldots,X_n$ are i.i.d. r.v.s. This implies $Y_{\alpha}(1)D_{\beta}(1)=Y_{\alpha}(1)X_1+\cdots + Y_{\alpha}(1)X_n$, which is not infinitely divisible since $Y_{\alpha}(1)X_1,\ldots, Y_{\alpha}(1)X_n$ are not independent r.v.s. Hence the RHS of \eqref{limit.GSTFSP} is not infinitely divisible, a contradiction. Thus the one-dimensional distributions of $\{\mathcal{S}_{\beta}^\alpha(t)\}_{t\geq 0}$ are not infinitely divisible. 
\end{proof}

The following result follows from Theorem \ref{id.GSTFSP} by taking $\lambda_j=0$ or $\mu_j=0$ for $j=1,\ldots,k$  in GSTFSP-II. 
\begin{corollary}
The one-dimensional distributions of GSTFCP are not infinitely divisible.
\end{corollary}

\begin{remark}\label{GSFSPlimit}
From Proposition \ref{limit.GSFSP-II.GFSP-II}, we observe that the one-dimensional distributions of GSFSP-II are infinitely divisible since $D_{\beta}(1)$ is infinitely divisible. Since GSFCP, SFPP and SFSPoK-II are special cases of GSFSP-II obtained by taking $\mu_j=0$ or $\lambda_j=0$; $k=1,\mu_j=0$ or $\lambda_j=0$ and $\lambda_j=\lambda,~\mu_j=\mu$ for $j=1,2\ldots,k$ respectively, the one-dimensional distributions of these processes are also infinitely divisible.  
\end{remark}

\begin{remark}\label{GFSPlimit}
From Proposition \ref{limit.GSFSP-II.GFSP-II}, we observe that the one-dimensional distributions of GFSP-II are not infinitely divisible since $Y_{\alpha}(1)$ is not infinitely divisible. Since GFCP, TFPP and FSPoK-II are special cases of GFSP-II obtained by taking $\mu_j=0$ or $\lambda_j=0$; $k=1,\mu_j=0$ or $\lambda_j=0$ and $\lambda_j=\lambda,~\mu_j=\mu$ for $j=1,2\ldots,k$ respectively, the one-dimensional distributions of these processes are also not infinitely divisible. 
\end{remark}

\section{Martingale Characterization}\label{sec 9}
In this section, we first obtain the weighted sum representations of the GSTFSP and its special cases, which will later be used to derive the martingale characterizations of these processes.  
\begin{theorem}\label{weighted_equality.gstfsp-I}
    Let $\{S_{1\beta_1,\beta_2}^{\alpha_1,\alpha_2}(t)\}_{t\geq 0}$, $\{S_{2\beta_1,\beta_2}^{\alpha_1,\alpha_2}(t)\}_{t\geq 0}$, $\dots$,$\{S_{k\beta_1,\beta_2}^{\alpha_1,\alpha_2}(t)\}_{t\geq 0}$ be $k$ independent counting processes. Then the process $\{\mathcal{S}_{\beta_1,\beta_2}^{\alpha_1,\alpha_2}(t)\}_{t\ge 0}$ is a GSTFSP-I if and only if each $\{S_{j\beta_1,\beta_2}^{\alpha_1,\alpha_2}(t)\}_{t\ge 0}$, $1\le j\le k$ is an independent STFSP-I such that $\mathcal{S}_{\beta_1,\beta_2}^{\alpha_1,\alpha_2}(t)=\sum_{j=1}^{k}j S_{j\beta_1,\beta_2}^{\alpha_1,\alpha_2}(t)$.
\end{theorem}
\begin{proof}Let $\{N_{1j}(t)\}_{t\ge 0}$ and $\{N_{2j}(t)\}_{t\ge 0}$, $j=1,\dots,k$, be two collections of counting processes such that all processes within each collection are mutually independent and the two collections themselves are independent. By Theorem~2.2 of \citet{Dhillon2024}, the process $\{\sum_{j=1}^{k} j\,N_{1j}(t)\}_{t\ge 0}$ is a GCP if and only if each $\{N_{1j}(t)\}_{t\ge 0}$ is a Poisson process with rate $\lambda_j$. Similarly, the process $\{\sum_{j=1}^{k} j\,N_{2j}(t)\}_{t\ge 0}$ is a GCP if and only if each $\{N_{2j}(t)\}_{t\ge 0}$ is a Poisson process with rate $\mu_j$. Let $\{D_{\beta_1}(Y_{\alpha_1}(t))\}_{t\ge 0}$ be a subordinator independent of $\{N_{1j}(t)\}_{t\ge 0}$ and $\{D_{\beta_2}(Y_{\alpha_2}(t))\}_{t\ge 0}$ be a subordinator independent of $\{N_{2j}(t)\}_{t\ge 0}$. Consequently, the time-changed process $\{M_{1\beta_1}^{\alpha_1}(t)=\sum_{j=1}^{k} j\,N_{1j}(D_{\beta_1}(Y_{\alpha_1}(t)))\}_{t\ge 0}$ is a GSTFCP if and only if each $\{N_{1j}(D_{\beta_1}(Y_{\alpha_1}(t)))\}_{t\ge 0}$ is a STFPP with rate $\lambda_j$. Similarly, the process $\{M_{2\beta_2}^{\alpha_2}(t)=\sum_{j=1}^{k} j\,N_{2j}(D_{\beta_2}(Y_{\alpha_2}(t)))\}_{t\ge 0}$ is a GSTFCP if and only if each $\{N_{2j}(D_{\beta_2}(Y_{\alpha_2}(t)))\}_{t\ge 0}$ is a STFPP with rate $\mu_j$. Therefore, the process $\{\mathcal{S}_{\beta_1,\beta_2}^{\alpha_1,\alpha_2}(t)=M_{1\beta_1}^{\alpha_1}(t)-M_{2\beta_2}^{\alpha_2}(t)=\sum_{j=1}^{k} j\,N_{1j}(D_{\beta_1}(Y_{\alpha_1}(t)))-\sum_{j=1}^{k} j\,N_{2j}(D_{\beta_2}(Y_{\alpha_2}(t)))=\sum_{j=1}^{k} j\,S_{j\beta_1,\beta_2}^{\alpha_1,\alpha_2}(t)\}_{t\ge 0}$, is a GSTFSP-I if and only if each process $\{S_{j\beta_1,\beta_2}^{\alpha_1,\alpha_2}(t)\}_{t\ge 0}$ is an independent STFSP-I. .
\end{proof}

\begin{theorem}\label{weighted_equality.gstfsp-II}
    Let $\{S_{1\beta}^{\alpha}(t)\}_{t\geq 0}$, $\{S_{2\beta}^{\alpha}(t)\}_{t\geq 0}$, $\dots$,$\{S_{k\beta}^{\alpha}(t)\}_{t\geq 0}$ be $k$ independent counting processes. Then the process $\{\mathcal{S}_{\beta}^{\alpha}(t)\}_{t\ge 0}$ is a GSTFSP-II if and only if each $\{S_{j\beta}^\alpha(t)\}_{t\ge 0}$, $1\le j\le k$ is an independent STFSP-II such that $\mathcal{S}_{\beta}^{\alpha}(t)=\sum_{j=1}^{k}j S_{j\beta}^\alpha(t)$.
\end{theorem}
\begin{proof}Let $\{N_{1j}(t)\}_{t\ge 0}$ and $\{N_{2j}(t)\}_{t\ge 0}$, $j=1,\dots,k$, be two collections of counting processes such that all processes within each collection are mutually independent and the two collections are independent of each other. Let $\{D_{\beta}(Y_{\alpha}(t))\}_{t\ge 0}$ be a subordinator independent of $\{N_{1j}(t)\}_{t\ge 0}$ and $\{N_{2j}(t)\}_{t\ge 0}$. Using similar arguments as in Theorem \ref{weighted_equality.gstfsp-I}, the result follows.
\end{proof}

The next result follows from Theorems \ref{weighted_equality.gstfsp-I} and \ref{weighted_equality.gstfsp-II}.
\begin{proposition}\label{weighted.proposition}
The GSFSP-I, GSFSP-II, GFSP-I, GFSP-II, GSP, GSTFCP, GSFCP and GFCP are equal to the weighted sum of $k$ independent SFSP-Is, SFSP-IIs, FSP-Is, FSP-IIs, Skellam processes, STFPPs, SFPPs and TFPPs respectively. In particular, if the rates of the processes are equal, then STFSPoK-I, STFSPoK-II, SFSPoK-I, SFSPoK-II, FSPoK-I, FSPoK-II, SPoK, STFPPoK, SFPPoK, FPPoK and PPoK are equal to the weighted sum of $k$ independent STFSP-Is, STFSP-IIs, SFSP-Is, SFSP-IIs, FSP-Is, FSP-IIs, Skellam Processes, STFPPs, SFPPs, FSPs and Poisson Processes.
\end{proposition}

\subsection{Martingale Characterization}

Consider a complete probability space $(\Omega,\mathcal{F},\mathbb{P})$. Recall that a $\mathbb{P}$-integrable and $\mathcal{F}_t$-adapted stochastic process $\{X(t)\}_{t\ge 0}$
is a $\mathcal{F}_t$-martingale if $\mathbb{E}(X(t)|\mathcal{F}_s)=X(s),\,0\le s\le t$, a.s. where $\{\mathcal{F}_t\}$ is a non-decreasing sequence of sub-sigma fields of $\mathcal{F}$. Also a point process is simple if it has unit jumps and locally finite if all jumps are finite in a bounded region. In 1964, \citet{Watanabe1964} provided a martingale characterization of the homogeneous Poisson process called the Watanabe characterization as shown below.
\begin{theorem}\label{Watanabe_Poisson}
    Let $\{N(t)\}_{t\ge 0}$ be a $\mathcal{F}_t$-adapted simple locally finite point process. Then, $\{N(t)\}_{t\ge 0}$ is a homogeneous Poisson process with intensity $\lambda>0$ \textit{iff} $\{N(t)-\lambda t\}_{t\ge 0}$ is a $\mathcal{F}_t$-martingale for some $\lambda>0$.
\end{theorem}
\citet{Aletti2018} extended the Watanabe characterization of a Poisson process to the TFPP. \citet{Dhillon2024} further investigated the Watanabe characterizations of the GCP and GFCP. \citet{Tathe2025} provided such characterizations for some non-homogeneous counting processes. However, note that this characterization does not apply to SFPP and hence STFPP since their moments do not exist. So we will investigate the Watanabe characterization of the tempered space fractional Poisson process (TSFPP) and its time fractional version TSTFPP, which have finite moments and are defined as
\begin{equation*}
N_{\beta,\theta}(t):=N(D_{\beta,\theta}(t)),\quad N^{\alpha}_{\beta,\theta}(t):=N(D_{\beta,\theta}(Y_{\alpha}(t)), \quad 
\end{equation*}
respectively, where the tempered stable subordinator $\{ D_{\beta,\theta}((t))\}_{t\ge 0}$, the inverse $\alpha$-stable subordinator $\{Y_{\alpha}(t)\}_{t\ge 0}$ and the Poisson process $\{N(t)\}_{t\ge 0}$ are independent of each other. We will use the following lemma (see Lemma 1 of \citet{Aletti2018}).
\begin{lemma}\label{martingale_lemma}
    Let $X$ be a right-continuous martingale. If $T$ and $S$ are stopping times such that $P(T<\infty)=1$ and $\{X(t\wedge T),t\ge 0\}$ is uniformly integrable, then $\mathbb{E}(X(T)|\mathcal{F}_{S\wedge T})=X(S\wedge T)$.
\end{lemma}

The following result provides the Watanabe characterization of a TSFPP.

\begin{theorem}\label{Watanabe_characterization_TSFPP}
    Let $\{Y(t)\}_{t\ge 0}$ be a simple locally finite point process. Then, $\{Y(t)\}_{t\ge 0}$ is a TSFPP \textit{iff} there exist a TSS $\{D_{\beta,\theta}(t)\}_{t\ge 0}$ and a constant $\lambda>0$ such that the process 
    \begin{equation*}
    \{X(t)\}_{t\ge 0}=\{Y(t)-\lambda D_{\beta,\theta}(t)\}_{t\ge 0}    
    \end{equation*} 
    is a right-continuous martingale with respect to the induced filtration $\mathcal{F}_t=\sigma(Y(s),s\le t)\vee\sigma(D_{\beta,\theta}(s), s\ge 0)$ and for any $T>0$,
    \begin{equation}\label{stopping_times_martingale_TSFPP}
        \{X(\tau),\,\tau\,\, \text{is stopping time such that}\,\,D_{\beta,\theta}(\tau)\le T\}
    \end{equation}
    is uniformly integrable.
\end{theorem}
\begin{proof}
    Let $\{Y(t)\}_{t\ge 0}$ be a TSFPP so that $Y(t)=N(D_{\beta,\theta}(t))$, where $\{D_{\beta,\theta}(t)\}_{t\ge 0}$ is the TSS independent of the Poisson process $\{N(t)\}_{t\ge 0}$ with rate $\lambda>0$.
    From \citet{Gupta2020a}), we have
    \begin{align}\label{TSFPP_moments}
       \mathbb{E}[D_{\beta,\theta}(t)] &= \beta\theta^{\beta-1}t,\quad\mathbb{V}[(D_{\beta,\theta}(t)]=\beta(1-\beta)\theta^{\beta-2}t,~  \nonumber \\
       \mathbb{E}[N(D_{\beta,\theta}(t))]&=\lambda\beta \theta^{\beta-1}t,\quad\mathbb{V}[N(D_{\beta,\theta}(t))]=\lambda\beta\theta^{\beta-1}t+\lambda^2\beta(1-\beta)\theta^{\beta-2}t,
    \end{align}
    which are all finite. Since $Y(t)$ and $D_{\beta,\theta}(t)$ are non-negative and monotone non-decreasing (by definition), and also bounded in $L^2$ (as implied by \eqref{TSFPP_moments}), the process $\{N(D_{\beta,\theta}(t))-\lambda D_{\beta,\theta}(t),\,\,0\le t\le T\}$ is uniformly integrable. Therefore, by Theorem \ref{Watanabe_Poisson} and Lemma \ref{martingale_lemma}, $\{N(D_{\beta,\theta}(t))-\lambda D_{\beta,\theta}(t)\}_{t\ge 0}$ is still a martingale.
   
   Next, let $\tau$ be a stopping time such that $D_{\beta,\theta}(\tau)\le T$ which implies $\lambda D_{\beta,\theta}(\tau)\le \lambda T$. Since $\{N(t)\}_{t\ge 0}$ is a Poisson process with intensity $\lambda>0$ and $\overset{\sim}{X}(t)=X(\tau\wedge t)$ is a martingale bounded in $L^2$ with $\overset{\sim}{X}(0)=0$ a.s., it follows that $\overset{\sim}{X}(t)$ converges to $X(\tau)$ in $L^2$ with variance bounded by
   \begin{equation*}
       \mathbb{E}(X^2(\tau))=\lim_{t\to\infty}\mathbb{E}(X^2(\tau\wedge t))\le\mathbb{V}(N(T))+\mathbb{V}(D_{\beta,\theta}(\tau))\le \text{constant}\times (1+T).
   \end{equation*}
   Therefore, the family \eqref{stopping_times_martingale_TSFPP} is uniformly bounded in $L^2$, and hence uniformly integrable.

   Conversely, it is enough to prove that $Y(t)=N(D_{\beta,\theta}(t))$, where $\{N(t)\}_{t\ge 0}$ is a Poisson process, independent of the TSS $\{D_{\beta,\theta}(t)\}_{t\ge 0}$ with rate $\lambda>0$. Let $Z(t)=\inf\{s:D_{\beta,\theta}(s)\ge t\}$ be the inverse of $\{D_{\beta,\theta}(t)\}_{t\ge 0}$. It can be observed that $\{Z(t),t\ge 0\}$ forms a family of stopping times. Therefore, by Lemma \ref{martingale_lemma}, 
   \begin{equation*}
       X(Z(t))=Y(Z(t))-\lambda D_{\beta,\theta}(Z(t))
   \end{equation*} is still a martingale. As $D_{\beta,\theta}(.)$ is continuous, we have $D_{\beta,\theta}(Z(t))=t$ which implies that $Y(Z(t))-\lambda t$ is a martingale. Furthermore, as $Z(t)$ is an increasing process, $Y(Z(t))$ is a simple point process. By Theorem \ref{Watanabe_Poisson}, it follows that $Y(Z(t))$ is a classical Poisson process with rate $\lambda>0$. If we define $N(t)=Y(Z(t))$, then the process $\{Y(t)=N(D_{\beta,\theta}(t))\}_{t\ge 0}$ represents a TSFPP. This proves the theorem.
\end{proof}
The next result provides the martingale characterization of a Tempered Generalized Space Time Fractional Counting Process (TGSTFCP).
\begin{theorem}\label{martingale characterization_TGSTFCP}
The process $\{M^{\alpha}_{\beta,\theta}(t)\}_{t\ge 0}$ is a TGSTFCP {\it iff} there exist $k$ simple locally finite point processes $\{N_{j}(t)\}_{t\ge 0}$ with rates $\lambda_j>0$ such that $M^{\alpha}_{\beta,\theta}(t)=\sum_{j=1}^k jN_{j}(t)$ and $\{X_j(t)=N_{j}(t)-\lambda_jD_{\beta,\theta}(Y_{\alpha}(t))\}_{t\ge 0}$ is a right-continuous $\{\mathcal{F}^j_t\}_{t\ge 0}$-martingale with respect to the induced filtration $\mathcal{F}^j_t=\sigma(N_{j}(s),s\le t)\vee\sigma (D_{\beta,\theta}(Y_{\alpha}(s)), s\ge 0)$. Moreover for any $T>0$, $\{X_j(\tau),\,\tau\,\, \text{is stopping time such that}\,\,D_{\beta,\theta}(Y_{\alpha}(\tau))\le T\}$ is uniformly integrable. 
\end{theorem}

\begin{proof}
Suppose the process $\{M^{\alpha}_{\beta,\theta}(t)\}_{t\ge 0}$ is a TGSTFCP. Let $\{N_1(t)\}_{t\ge 0},\,\{N_2(t)\}_{t\ge 0},\,\ldots,\{N_k(t)\}_{t\ge 0}$ be $k$ independent locally finite point processes with rates $\lambda_1>0,\,\lambda_2>0,\ldots,\lambda_k>0$ respectively and $\{D_{\beta,\theta}(t)\}_{t\ge 0}$ be a TSS independent of these processes. Using Proposition \ref{weighted.proposition}, we obtain $N_j(t)=N^{\alpha}_{j\beta,\theta}(t)$ and $M^{\alpha}_{\beta,\theta}(t)=\sum_{j=1}^{k}jN^{\alpha}_{j\beta,\theta}(t)$ where $\{N^{\alpha}_{j\beta,\theta}(t)\}_{t\ge 0}$ are TSTFPPs. Consequently, Theorem \ref{Watanabe_characterization_TSFPP} implies that each $\{X_j(t)\}_{t\ge 0}$ is a martingale with respect to $\{\mathcal{F}^j_t\}_{t\ge 0}$ and $\{X_j(\tau):\,D_{\beta,\theta}(Y_{\alpha}(\tau))<T\}$, where $\tau$ is a stopping time, is uniformly integrable. Conversely, it suffices to establish that  $\{N_{j}(t)\}_{t\ge 0}$ is a TSTFPP for each $j=1,2,\dots,k$. However, this directly follows from Theorem \ref{Watanabe_characterization_TSFPP}. Therefore from Proposition \ref{weighted.proposition} , it follows that $M^{\alpha}_{\beta,\theta}=\sum_{j=1}^k jN_{j}(t)$ is a TGSTFPP. This completes the proof.    
\end{proof}

Note that a TGSTFSP can be constructed either as the difference of two independent TGSTFCPs, referred to as TGSTFSP-I, or alternatively by time-changing a GSP with an independent subordinator $\{D_{\beta,\theta}(Y_\alpha(t))\}_{t\ge 0}$, referred to as TGSTFSP-II. The following result establishes the martingale characterization of TGSTFSP-I.
\begin{proposition}\label{martingale_characterization_TGSTFSP-I}
The process $\{\mathcal{S}^{\alpha_1,\alpha_2}_{\beta_1,\theta_1,\beta_2,\theta_2}(t)\}_{t\ge 0}$ is a TGSTFSP-I if and only if there exist $k$ simple locally finite point processes $\{S^{\alpha_1,\alpha_2}_{j\beta_1,\theta_1,\beta_2,\theta_2}(t)\}_{t\ge 0}$, $1\le j\le k$ such that $\mathcal{S}^{\alpha_1,\alpha_2}_{\beta_1,\theta_1,\beta_2,\theta_2}(t)=\sum_{j=1}^{k}jS^{\alpha_1,\alpha_2}_{j\beta_1,\theta_1,\beta_2,\theta_2}(t)$ and $X_j(t)=S^{\alpha_1,\alpha_2}_{j\beta_1,\theta_1,\beta_2,\theta_2}(t)-(\lambda_jD_{\beta_1,\theta_1}(Y_{\alpha_1}(t))-\mu_jD_{\beta_2,\theta_2}(Y_{\alpha_2}(t))$, $t\ge 0$ is a right continuous $\{\mathcal{F}_t^j\}_{t\ge 0}$-martingale, where $\mathcal{F}_t^j=\sigma(\mathcal{S}^{\alpha_1,\alpha_2}_{j\beta_1,\theta_1,\beta_2,\theta_2}(s),s\le t)\vee\sigma(D_{\beta_1,\theta_1}(Y_{\alpha_1}(s)),s\ge 0)\vee\sigma(D_{\beta_2,\theta_2}(Y_{\alpha_2}(s)),s\ge 0)$, $\{D_{\beta_1,\theta_1}(t)\}_{t\ge 0}$, $\{D_{\beta_2,\theta_2}(t)\}_{t\ge 0}$ are tempered stable subordinators, $\{Y_{\alpha_1}(t)\}_{t\ge 0}$, $\{Y_{\alpha_2}(t)\}_{t\ge 0}$ are inverse stable subordinators, and all the subordinators are independent of each other. Moreover, for any $T>0$, \\$\{X_j(\tau),\tau\,\, \text{is stopping time such that}\,\,\max(Y_{\alpha_1}(\tau),~Y_{\alpha_2}(\tau))\le T\}$ is uniformly integrable.
\end{proposition}
\begin{proof}
Let $\{M^{\alpha_1}_{1\beta_1,\theta_1}(t)\}_{t\ge 0}$ and $\{M^{\alpha_2}_{2\beta_2,\theta_2}(t)\}_{t\ge 0}$ be two independent TGSTFCPs. By definition, we have $\mathcal{S}^{\alpha_1,\alpha_2}_{\beta_1,\theta_1,\beta_2,\theta_2}(t)=M(D_{\beta_1,\theta_1}(Y_{\alpha_1}(t)))-M(D_{\beta_2,\theta_2}(Y_{\alpha_2}(t)))=\sum_{j=1}^{k}j(N_{1j}(D_{\beta_1,\theta_1}(Y_{\alpha_1}(t)))-N_{2j}(D_{\beta_2,\theta_2}(Y_{\alpha_2}(t))))=\sum_{j=1}^{k}j(N^{\alpha_1}_{1j\beta_1,\theta_1}(t)-N^{\alpha_2}_{2j\beta_2,\theta_2}(t))=\sum_{j=1}^{k}jS^{\alpha_1,\alpha_2}_{j\beta_1,\theta_1,\beta_2,\theta_2}(t)$, where $\{S^{\alpha_1,\alpha_2}_{j\beta_1,\theta_1,\beta_2,\theta_2}(t)\}_{t\ge 0}$ is a TSTFSP. The martingale characterization of a TSTFPP can be obtained from Theorem \ref{Watanabe_characterization_TSFPP} by replacing $D_{\beta,\theta}(t)$ with $D_{\beta,\theta}(Y_{\alpha}(t))$. Since the difference of two martingales with respect to a common filtration is a martingale, a TSTFSP-I is a martingale whose characterization can be obtained from that of a TSTFPP. The result follows from the fact that a weighted sum of martingales (here $\{S^{\alpha_1,\alpha_2}_{j\beta_1,\theta_1,\beta_2,\theta_2}(t)\}_{t\ge 0}$) with respect to a common filtration (here $\{\mathcal{F}_t^j\}_{t\ge 0}$) is a martingale.
\end{proof}
The next result provides the martingale characterization of TGSTFSP-II.
\begin{proposition}\label{martingale_characterization_TGSTFSP-II}
The process $\{\mathcal{S}^{\alpha}_{\beta,\theta}(t)\}_{t\ge 0}$ is a TGSTFSP-II if and only if there exist $k$ simple locally finite point processes $\{S^{\alpha}_{j\beta,\theta}(t)\}_{t\ge 0}$, $1\le j\le k$ such that $\mathcal{S}^{\alpha}(t)=\sum_{j=1}^{k}jS^{\alpha}_{j\beta,\theta}(t)$ and $X_j(t)=S^{\alpha}_{j\beta,\theta}(t)-(\lambda_j-\mu_j)D_{\beta,\theta}(Y_{\alpha}(t))$, $t\ge 0$ is a right continuous $\{\mathcal{F}_t^j\}_{t\ge 0}$-martingale, where $\{D_{\beta,\theta}(t)\}_{t\ge 0}$ and $\{Y_{\alpha}(t)\}_{t\ge 0}$ are independent tempered stable subordinator and inverse $\alpha$-stable subordinator, repsectively that are independent of $\{S_j(t)\}_{t\ge 0}$, and $\mathcal{F}_t^j=\sigma(\mathcal{S}^{\alpha}_{j\beta,\theta}(s),s\le t)\vee\sigma(D_{\beta,\theta}(Y_{\alpha}(s)),s\ge 0)$. Moreover, for any $T>0$, $\{X_j(\tau),\tau\,\, \text{is stopping time such that}\,\,D_{\beta,\theta}(Y_\alpha(\tau))\le T\}$ is uniformly integrable.
\end{proposition}
\begin{proof}
    We know that a GSP can be expressed as the difference of two independent GCPs, each of which admits a unique representation as a weighted sum of $k$ independent Poisson processes by Proposition \ref{weighted.proposition}. Specifically, the GSP can be written as $\mathcal{S}(t)=\sum_{j=1}^{k} jS_j(t)$, where $S_{j}(t)=N_{1j}(t)-N_{2j}(t)$, and $\{N_{1j}(t)\}_{t\ge0}$ and $\{N_{2j}(t)\}_{t\ge0}$ are independent Poisson processes with rates $\lambda_j$ and $\mu_j$ respectively. Then a TGSTFSP-II $\{\mathcal{S}^{\alpha}_{\beta,\theta}(t)\}_{t\ge 0}$ is given by $\mathcal{S}^{\alpha}_{\beta,\theta}(t)=\mathcal{S}(D_{\beta,\theta}(Y_{\alpha}(t)))=\sum_{j=1}^{k} jS_j(D_{\beta,\theta}(Y_{\alpha}(t)))=\sum_{j=1}^{k} jS_{j\beta,\theta}^{\alpha}(t)$, where $\{S_{j\beta,\theta}^{\alpha}(t)\}_{t\ge 0}$ is a TSTFSP-II. Since $\{N_{1j}(t)-\lambda_jt\}_{t\ge0}$ and $\{N_{2j}(t)-\mu_jt\}_{t\ge0}$ are martingales, $\{S_j(t)-(\lambda_j-\mu_j)t\}_{t\ge 0}$ is a martingale. It follows that $\{S_j(D_{\beta,\theta}(Y_{\alpha}(t)))-(\lambda_j-\mu_j)D_{\beta,\theta}(Y_{\alpha}(t))\}_{t\ge 0}$ is a right-continuous martingale with respect to the filtration $\mathcal{F}_t^j=\sigma(S^{\alpha}_{j\beta,\theta}(s),s\le t)\vee\sigma(D_{\beta,\theta}(Y_{\alpha}(s)),s\ge0)$. The uniform integrability of $\{X_j(\tau)\}$, for stopping times $\tau$ satisfying $D_{\beta,\theta}(Y_\alpha(\tau))\le T$, follows from the corresponding property established in Theorem~\ref{Watanabe_characterization_TSFPP} applied to a Skellam process and $D_{\beta,\theta}(\tau)$ replaced by $D_{\beta,\theta}(Y_{\alpha}(\tau))$.

Conversely, suppose there exist simple locally finite point processes $\{S^{\alpha}_{j\beta,\theta}(t)\}_{t\ge0}$ and a subordinator $\{D_{\beta,\theta}(Y_{\alpha}(t))\}_{t\ge 0}$ such that each $\{X_j(t)=S^{\alpha}_{j\beta,\theta}(t)-(\lambda_j-\mu_j)D_{\beta,\theta}(Y_{\alpha}(t))\}_{t\ge 0}$ is a right-continuous $\mathcal{F}_t^j$-martingale. For $\alpha=\beta=\theta=1$, $\{X_j(t)=S_{j}(t)-(\lambda_j-\mu_j)t\}_{t\ge 0}$ is a right-continuous $\mathcal{F}_t^j$-martingale which implies that $\{S_j(t)\}_{t\ge 0}$ is a Skellam process, that is, each process $\{S^{\alpha}_{j\beta,\theta}(t)=S_{j}(D_{\beta,\theta}(Y_{\alpha}(t)))\}_{t\ge 0}$ is a TSTFSP-II. Consequently, their weighted sum $\{\mathcal{S}^{\alpha}_{\beta,\theta}(t)=\sum_{j=1}^{k} j\,S^{\alpha}_{j\beta,\theta}(t)\}_{t\ge 0}$ is a TGSTFSP-II. This completes the proof.
\end{proof}

\section{Running average processes}\label{sec 11}
In this section, we introduce the running-average processes of GSFSP-I, GSFSP-II, and GSFCP, and study their properties.
\subsection{Running Average of GSFSP-I}
\begin{definition}\label{def_running_average_GSFSP-I}
We define the running average process of a GSFSP-I $\{\mathcal{S}_{\beta_1,\beta_2}(t)\}_{t\ge 0}$ by taking the time-scaled integral of its path (see \citet{Xia2018, Gupta2020}) as follows:
\begin{equation*}
\mathcal{S}^A_{\beta_1,\beta_2}(t)=\frac{1}{t}\int_{0}^{t}\mathcal{S}_{\beta_1,\beta_2}(s)ds. \label{78}
\end{equation*}
\end{definition}
Note that $\{\mathcal{S}^A_{\beta_1,\beta_2}(t)\}_{t \geq 0}$ satisfies the following differential equation with initial condition $\mathcal{S}^A_{\beta_1,\beta_2}(0)=0$:
\begin{equation*}
    \frac{d}{dt}(\mathcal{S}^A_{\beta_1,\beta_2}(t))=\frac{1}{t}\mathcal{S}_{\beta_1,\beta_2}(t)-\frac{1}{t^2}\int_{0}^{t}\mathcal{S}_{\beta_1,\beta_2}(s)ds    
\end{equation*}
which shows that it has continuous sample paths of bounded total variation. 
The following result provides the distribution of the running average process of a GSFSP-I in terms of a compound Poisson process. 
\begin{theorem}\label{compound.gsfsp-I}
Let $\{Y(t)\}_{t\ge 0}$ be a compound Poisson process given by
\begin{equation*}
   Y(t)=\sum_{i=1}^{N(D_{\beta}(t))}X_i
\end{equation*}
where $\{N(D_{\beta}(t))\}_{t\ge 0}$ is a SFPP with intensity parameter $(\Lambda+T)$ and $X_i$s are i.i.d. random variables, independent of $N(D_{\beta}(t))$, with density
\begin{equation*}
f(x) = \displaystyle\frac{1}{2\pi}\displaystyle\int_{-\infty}^{\infty}e^{-iux}[1-\frac{1}{\Lambda+T}(\int_{0}^{1}\{(\sum_{j=1}^{k}(1-e^{iuzj})\lambda_j)^{\beta_1} + (\sum_{j=1}^{k}(1-e^{-iuzj})\mu_j)^{\beta_2}\}dz)^{1/\beta}]du.
\end{equation*}
Then
\begin{equation*}
Y(t) \overset{d}{=} \mathcal{S}^A_{\beta_1,\beta_2}(t).
\end{equation*}
\end{theorem}
\begin{proof}
Let us denote $K(t)=\int_{0}^{t}\mathcal{S}_{\beta_1,\beta_2}(s)ds$. Since $\{\mathcal{S}_{\beta_1,\beta_2}(t)\}_{t\ge 0}$ is a L\'{e}vy process, using Lemma \ref{running average lemma}, the characteristic function of $K(t)$ is given as
\begin{equation*}
    \phi_{K(t)}(u)=e^{t(\int_{0}^{1}\log \phi_{\mathcal{S}_{\beta_1,\beta_2}(1)}(tuz)dz)}.
\end{equation*}
Substituting $\alpha=1$ in the p.g.f. of GSTFSP-I (see Proposition \ref{p.g.f._GSTFSP1}), the characteristic function of $\{\mathcal{S}_{\beta_1,\beta_2}(t)\}_{t\ge 0}$ is
\begin{equation*}
\phi_{\mathcal{S}_{\beta_1,\beta_2}(t)}(u) = \exp[-t\{(\sum_{j=1}^{k}(1-e^{iuj})\lambda_j)^{\beta_1} + (\sum_{j=1}^{k}(1-e^{-iuj})\mu_j)^{\beta_2}\}]
\end{equation*}
which implies
\begin{equation*}
\int_{0}^{1}\log \phi_{\mathcal{S}_{\beta_1,\beta_2}(1)}(tuz)dz=-\int_{0}^{1}\{(\sum_{j=1}^{k}(1-e^{ituzj})\lambda_j)^{\beta_1} + (\sum_{j=1}^{k}(1-e^{-ituzj})\mu_j)^{\beta_2}\}dz.
\end{equation*}
Hence the characteristic function of $\{\mathcal{S}^A_{\beta_1,\beta_2}(t)\}_{t\ge 0}$ is 
\begin{equation*}
\phi_{\mathcal{S}^{A}_{\beta_1,\beta_2}(t)}(u)=\phi_{K(t)/t}(u)=\phi_{K(t)}(u/t)=e^{t(\int_{0}^{1}\log \phi_{\mathcal{S}_{\beta_1,\beta_2}(1)}(tuz)dz)}=e^{-t\int_{0}^{1}\{(\sum_{j=1}^{k}(1-e^{iuzj})\lambda_j)^{\beta_1} + (\sum_{j=1}^{k}(1-e^{-iuzj})\mu_j)^{\beta_2}\}dz}.
\end{equation*}
Since $X_i$s are i.i.d. random variables with density given by
\begin{equation*}
f(x) = \displaystyle\frac{1}{2\pi}\displaystyle\int_{-\infty}^{\infty}e^{-iux}[1-\frac{1}{\Lambda+T}(\int_{0}^{1}\{(\sum_{j=1}^{k}(1-e^{iuzj})\lambda_j)^{\beta_1} + (\sum_{j=1}^{k}(1-e^{-iuzj})\mu_j)^{\beta_2}\}dz)^{1/\beta}]du,
\end{equation*}
the characteristic function of $X_1$ using inversion theorem is
\begin{equation*}
\phi_{X_1}(u) = 1-\frac{1}{\Lambda+T}(\int_{0}^{1}\{(\sum_{j=1}^{k}(1-e^{iuzj})\lambda_j)^{\beta_1} + (\sum_{j=1}^{k}(1-e^{-iuzj})\mu_j)^{\beta_2}\}dz)^{1/\beta}. 
\end{equation*}
Next, the characteristic function of $Y(t)$ is
\begin{align*}
    \phi_{Y(t)}(u)
    =\mathbb{E}[\exp(iu\sum_{i=1}^{N(D_{\beta}(t))}X_i)]
    &= \mathbb{E}[\mathbb{E}(\exp(iu\sum_{i=1}^{N(D_{\beta}(t))}X_i)\big|N(D_{\beta}(t)))]\\
    &= \sum_{n=0}^{\infty}\mathbb{E}(\exp(iu\sum_{i=1}^{n}X_i))P(N(D_{\beta}(t))=n) \\
    &=\sum_{n=0}^{\infty}(\phi_{X_1}(u))^n P(N(D_{\beta}(t))=n) \,\,= P^*(\phi_{X_1}(u))   
\end{align*}
where $P^*(.)$ is the p.g.f. of $\{N(D_{\beta}(t))\}_{t\ge 0}$ with intensity parameter $(\Lambda+T)t$. Substituting $\alpha=1$ and $k=1$ in the p.g.f. of GSTFCP (see \citet{Katariaarxiv}), we have $P^*(s)=\exp(-t((1-s(\Lambda+T))^\beta)$. Hence
\begin{equation*}
\phi_{Y(t)}(u)=P^*(\phi_{X_1}(u)) = e^{-t((1-\phi_{X_1}(u))(\Lambda+T))^\beta}=e^{-t\int_{0}^{1}\{(\sum_{j=1}^{k}(1-e^{iuzj})\lambda_j)^{\beta_1} + (\sum_{j=1}^{k}(1-e^{-iuzj})\mu_j)^{\beta_2}\}dz}
    = \phi_{\mathcal{S}^A_{\beta_1,\beta_2}(t)}(u),
\end{equation*}
thereby proving the result.
\end{proof} 
\begin{remark}
    The running average process of GSFSP-II can be defined similarly to Definition \ref{def_running_average_GSFSP-I}. Since GSFSP-II is a L\'{e}vy process, its compound Poisson representation can be obtained using Lemma \ref{running average lemma} and the characteristic function of GSTFSP-II (see Proposition \ref{p.g.f._GSTFSP2}) similar to Theorem \ref{compound.gsfsp-I}.
\end{remark}

\begin{corollary}
If $\beta_1=\beta_2=1$, then the running average process of GSFSP-I and hence its compound Poisson representation reduce to those of GSP (see \citet{Tathe2024}). If $\beta_1=\beta_2=k=1$, then the running average process of GSFSP-I and hence its compound Poisson representation reduce to those of a Skellam process. However, if $\beta_1=\beta_2=1$, $\lambda_1=\cdots=\lambda_k$ and $\mu_1=\cdots=\mu_k$, then we obtain the running average process and its compound Poisson representation of a Skellam process of order $k$ (see \citet{Gupta2020}).
\end{corollary}

\begin{remark}
As GSTFSP-I and GSTFSP-II are not L\'{e}vy processes due to dependent increments, Lemma \ref{running average lemma} can't be used to find the distribution of their running average processes.
\end{remark}

\subsection{Running Average of GSFCP}
Let $\{M_{\beta}(t)\}_{t\ge 0}$ be a GSFCP with intensity parameters $\lambda_1,\lambda_2,\ldots,\lambda_k$.
Its running average process $\{M^A_{\beta}(t)\}_{t \geq 0}$ can be defined by taking the time-scaled integral of its path similar to that of a GSFSP. It follows that $\{M^A_{\beta}(t)\}_{t \geq 0}$ has continuous sample paths of bounded total variation.   

The next result shows the distribution of the running average process of a GSFCP.

\begin{theorem}\label{th.gcp}
The running average process of GSFCP has a compound Poisson representation given by
\begin{equation*}
M^A_{\beta}(t) \overset{d}{=}\sum_{i=1}^{N(D_{\beta}(t))}X_i
\end{equation*}
where $\{N(D_{\beta}(t))\}_{t\ge 0}$ is a SFPP with intensity parameter $\Lambda$ and $X_i$s are i.i.d. random variables, independent of $N(D_{\beta}(t))$, with density given by
\begin{equation*}
f(x) = \displaystyle\frac{1}{2\pi}\displaystyle\int_{-\infty}^{\infty}e^{-iux}[1-\frac{1}{\Lambda}(\int_{0}^{1}(\sum_{j=1}^{k}(1-e^{iuzj})\lambda_j)^\beta dz)^{1/\beta}]du.
\end{equation*}
\end{theorem} 

\begin{proof}
The proof is similar to that of Theorem \ref{compound.gsfsp-I} and hence omitted.
\end{proof}



\begin{corollary}
If $\beta=1$, then the running average process of a GSFCP and hence its compound Poisson representation  reduce to those of a GCP (see \citet{Tathe2024}). If $\beta=k=1$, then the running average process of a GSFCP and hence its compound Poisson representation reduce to those of a standard Poisson process (see \citet{Xia2018}), and to those of a Poisson process of order $k$ (see \citet{Gupta2020}) if $\lambda_1=\cdots=\lambda_k$. 
\end{corollary}
For the importance of running average processes in general, see Remark 13 of \citet{Tathe2024}.
\section{Simulation}\label{sec 12}
In this section, we present plots of simulated sample paths for GSTFSP-I and GSTFSP-II. We need to simulate two GSTFCPs first, and then take their difference to simulate GSTFSP-I. On the other hand, we need to simulate two GCPs first and take their difference to obtain a GSP, which can then be time-changed to simulate GSTFSP-II. For simulation of GCP, we use its compound Poisson representation given by
\begin{equation*}
M(t)\overset{d}{=}\sum_{i=1}^{N(t))}X_i 
\end{equation*}
where $X_i$s are i.i.d. random variables such that $P(X_1=j)=\frac{\lambda_j}{\Lambda}$ and $N(t)$ is an independent Poisson process with rate $\Lambda=\sum_{j=1}^{k}\lambda_j$. Similarly GSTFCP can be simulated using the following representation:
\begin{equation*}
M^{\alpha}_{\beta}(t)\overset{d}{=}\sum_{i=1}^{N(D_{\beta}(Y_{\alpha}(t)))}X_i 
\end{equation*}
where $X_i$s are as above and $N(D_{\beta}(Y_{\alpha}(t))$ is an independent STFPP with parameter $\Lambda t$. For a fixed time point $t$, GSTFSP-I has to be simulated at time point $D_\beta(Y_{\alpha}(t))$ while GSTFSP-II has to be simulated at $D_\beta(t)$. \\

{\bf Algorithm for simulation of GSTFSP-I} \\ 

\noindent
(a) Algorithm for simulation of inverse stable subordinators (see \citet{Maheshwari2019}): \\
\begin{table}[H]
\resizebox{14.5cm}{!}{
\begin{tabular}{@{}llll@{}}
\toprule
\multicolumn{4}{l}{\textbf{Input:} Choose the time fractional indices $\alpha_1$, $\alpha_2$ and $n$ uniformly spaced time points} \\ { $0=\tau_1,\ldots,\tau_n=T^*$ such that $\tau_{j}-\tau_{j-1}=h$ for $2\le j\le n$.} \\

\,\,\textbf{Output:} $Y_{\alpha_1}(t)$ and $Y_{\alpha_2}(t)$, the inverse stable subordinator values. \\ 
\midrule
\multicolumn{4}{l}{\begin{tabular}[c]{@{}l@{}}
\textit{Initialization:} $i=0$, $t_{1(0)}=0$, $t_{2(0)}=0$, $Q_{1(0)}=0$, $Q_{2(0)}=0$.\\
1: \textbf{while} $t_{1(i)}<T^*$ and $t_{2(i)}<T^*$ do\\

2: Generate increments $Q_{1(i+1)}$ and $Q_{2(i+1)}$ of two independent stable subordinators \\ $D_{\alpha_1}(t)$ and $D_{\alpha_2}(t)$ respectively as follows. \\
    \quad (a) Generate $U_1\sim U[0,\pi]$, $U_2\sim U[0,\pi]$, $V_1\sim \exp(1)$ and $V_2\sim \exp(1)$. \\
    \quad(b) Compute $Q_{1(i+1)}$ and $Q_{2(i+1)}$ for the time interval $(\tau_{j-1}, \tau_{j})$ as \\ 
    \quad $Q_{l(i+1)}=D_{\alpha_l}(\tau_j)-D_{\alpha_l}(\tau_{j-1})\stackrel{d}{=}D_{\alpha_l}(h)=h^{1/\alpha_l}\frac{(\sin \alpha_l U_l)(\sin (1-\alpha_l)U_l)^{(1-\alpha_l)/\alpha_l}}{(\sin U_l)^{1/\alpha_l}V_l^{(1-\alpha_l)/\alpha_l}}$, for $l=1,2$.\\
3:\, Set $Y_{l(i)}=Y_l(\ceil{t_{l(i)}/h} + 1)=\cdots= Y_l(\lfloor (t_{l(i)}+Q_{l(i+1)})/h\rfloor+1)=h*i$ for $l=1,2$.\\
4:\, $t_{1(i+1)}=t_{1(i)}+Q_{1(i+1)}$ and $t_{2(i+1)}=t_{2(i)}+Q_{2(i+1)}$, $i=i+1$.\\
5: \textbf{end while}\\
6: \textbf{return} $Y_{\alpha_1}(t_{1(i)})$ and $Y_{\alpha_2}(t_{2(i)})$.\\
\quad The discretized sample paths of $Y_{\alpha_1}(t)$ at $t_{1(i)}$ and $Y_{\alpha_2}(t)$ at $t_{2(i)}$ are $Y_{1(i)}$ and $Y_{2(i)}$ respectively.\\
\end{tabular}}       \\
\bottomrule 
\end{tabular}
}
\end{table}

\noindent
(b) Algorithm for simulation of time-changed stable subordinators: \\
\begin{table}[H]
\resizebox{14.5cm}{!}{
\begin{tabular}{@{}llll@{}}
\toprule
\multicolumn{4}{l}{\textbf{Input:} $Y_{\alpha_1}(t_{1(i)})$ and $Y_{\alpha_2}(t_{2(i)})$, the simulated inverse stable subordinator values. Choose the space}\\{fractional indices $\beta_1$ and $\beta_2$.} \\

\,\,\textbf{Output:} $D_{\beta_1}(Y_{\alpha_1}(t))$ and $D_{\beta_2}(Y_{\alpha_2}(t))$, the time-changed stable subordinator values. \\ 
\midrule
\multicolumn{4}{l}{\begin{tabular}[c]{@{}l@{}}
\textit{Initialization:} $i=0$, $t_{1(0)}=0$, $t_{2(0)}=0$, $Q_{1(0)}=0$, $Q_{2(0)}=0$.\\
1: \textbf{while} $t_{1(i)}<Y_{\alpha_1}(t_{1(i)})$ and $t_{2(i)}<Y_{\alpha_2}(t_{2(i)})$ do\\

2: Generate increments $Q_{1(i+1)}$ and $Q_{2(i+1)}$ of two independent stable subordinators\\ $D_{\beta_1}(t)$ and $D_{\beta_2}(t)$ respectively as follows. \\
    \quad (a) Generate $U_1\sim U[0,\pi]$, $U_2\sim U[0,\pi]$, $V_1\sim \exp(1)$ and $V_2\sim \exp(1)$. \\
    \quad(b) Compute $Q_{1(i+1)}$ and $Q_{2(i+1)}$ for the time intervals $(Y_{\alpha_1}(t_{1(i-1)}), Y_{\alpha_1}(t_{1(i)}))$\\ and $(Y_{\alpha_2}(t_{2(i-1)}), Y_{\alpha_2}(t_{2(i)}))$ respectively as \\ 
    \quad $Q_{l(i+1)}=D_{\beta_l}(Y_{\alpha_l}(t_{l(i)}))-D_{\beta_l}(Y_{\alpha_l}(t_{l(i-1)}))\stackrel{d}{=}D_{\beta_l}(h_i)=h_{l(i)}^{1/\beta_l}\frac{(\sin \beta_l U_l)(\sin (1-\beta)U_l)^{(1-\beta_l)/\beta_l}}{(\sin U_l)^{1/\beta_l}V_l^{(1-\beta_l)/\beta_l}}$;\\ \quad $h_{l(i)}=Y_{\alpha_l}(t_{l(i)})-Y_{\alpha_l}(t_{l(i-1)})$, for $l=1,2$.\\
3:\, $t_{l(i+1)}=t_{l(i)}+Q_{l(i+1)}$, $i=i+1$ for $l=1,2$.\\
4: \textbf{end while}\\
5: \textbf{return} $D_{\beta_1}(Y_{\alpha_1}(t_{1(i)}))$ and $D_{\beta_2}(Y_{\alpha_2}(t_{2(i)}))$.\\
\quad The discretized sample paths of $D_{\beta_1}(Y_{\alpha_1}(t))$ at $t_{1(i)}$ and $D_{\beta_2}(Y_{\alpha_2}(t))$ at $t_{2(i)}$ are $\sum_{k=0}^i Q_{1k}$\\ and $\sum_{k=0}^i Q_{2k}$ respectively. \\
\end{tabular}}       \\
\bottomrule 
\end{tabular}
}
\end{table}  
\noindent
(c) Algorithm for simulation of sample path (see Figure \ref{GSTFSP-I_sample_path}): \\
\begin{table}[H]
\resizebox{14.5cm}{!}{
\begin{tabular}{@{}llll@{}}
\toprule
\multicolumn{4}{l}{\begin{tabular}[c]{@{}l@{}}
\textbf{Input}: $D_{\beta_1}(Y_{\alpha_1}(t_{1(i)}))$ and $D_{\beta_2}(Y_{\alpha_2}(t_{2(i)}))$, the simulated time-changed stable subordinator values.\\ Choose rate parameters $\lambda_1, \lambda_2,...,\lambda_k$ for $\{M^{\alpha_1}_{1\beta_1}(t)\}_{t\geq 0}$ and $\mu_1, \mu_2,...,\mu_k$ for $\{M^{\alpha_2}_{2\beta_2}(t)\}_{t\geq 0}$ such that\\ $\Lambda=T$ where $\Lambda=\sum_{j=1}^k\lambda_j$ and $T=\sum_{j=1}^k\mu_j$. \\

\textbf{Output:} $\mathcal{S}^{\alpha_1,\alpha_2}_{\beta_1,\beta_2}(t)$, simulated sample path for GSTFSP-I.\\
\midrule
\textit{\quad Initialization: $t=0$, $Z_{1i}=Z_{2i}=0$, $M^{\alpha_1}_{1\beta_1}(t)=M^{\alpha_2}_{2\beta_2}(t)=0$.}\\

1: \textbf{while} $t<D_{\beta_1}(Y_{\alpha_1}(t_{1(i)}))$ and $t<D_{\beta_2}(Y_{\alpha_2}(t_{2(i)}))$ do\\



2: Using Monte-Carlo simulation, generate independent discrete random variables $X_{1i}$ and $X_{2i}$ with \\

\quad $P(X_{1i}=j)=\frac{\lambda_j}{\Lambda}$ and $P(X_{2i}=j)=\frac{\mu_j}{T}$ for $1\le j\le k$.\\



3: Set $Z_{1i}=Z_{1i}+X_{1i}$ and $Z_{2i}=Z_{2i}+X_{2i}$. \\



4: Generate a uniform random variable $U \sim U(0,1)$.\\

5. Set $t= t + (\frac{-1}{\Lambda}\ln{U})$.\\ 

6: \textbf{end while} \\

7: Set $M^{\alpha_1}_{1\beta_1}(t_i)=\sum_{j=1}^{i}Z_{1j}$ and $ M^{\alpha_2}_{2\beta_2}(t_i)=\sum_{j=1}^{i}Z_{2j} $\,. \\

8: Compute $\mathcal{S}^{\alpha_1,\alpha_2}_{\beta_1,\beta_2}(t_i)= M^{\alpha_1}_{1\beta_1}(t_i)-M^{\alpha_2}_{2\beta_2}(t_i)$. \\

9: \textbf{return} $\mathcal{S}^{\alpha_1,\alpha_2}_{\beta_1,\beta_2}(t_i)$.\\

\end{tabular}}       \\
\bottomrule 
\end{tabular}
}
\end{table}  
\begin{figure}[H] 
\centering
\includegraphics[scale=0.6]{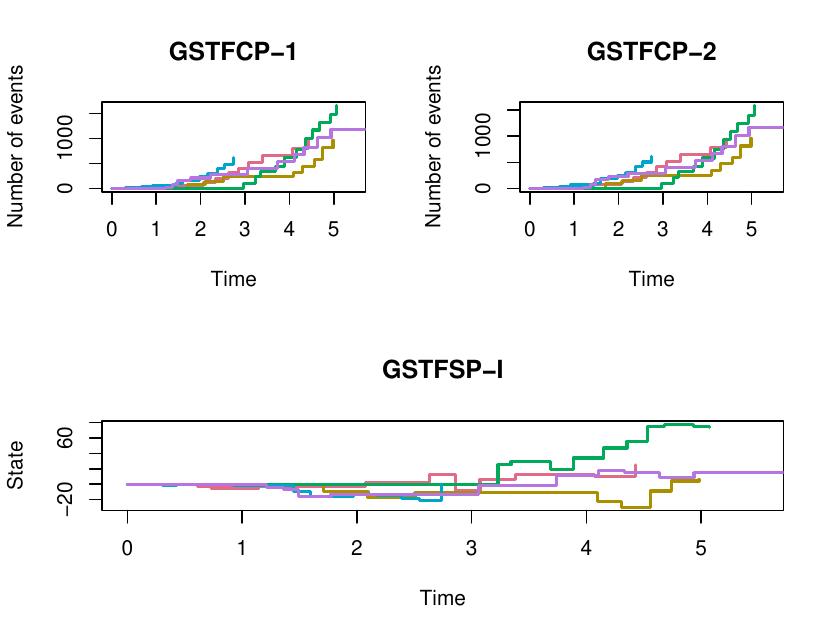}
\caption{Sample paths of GSTFCP-1 for $\alpha_1=0.5$, $\beta_1=0.8$, $\lambda_1=1$, $\lambda_2=3$, $\lambda_3=2$, $\lambda_4=2$, $\lambda_5=2$ and GSTFCP-2 for $\alpha_2=0.4$, $\beta_2=0.7$, $\mu_1=2$, $\mu_2=2$, $\mu_3=3$, $\mu_4=3$, $\mu_5=2$. The sample paths of GSTFSP-I are obtained as the difference between those of GSTFCP-1 and GSTFCP-2.}\label{GSTFSP-I_sample_path}
\end{figure}
Figure~\ref{GSTFSP-I_sample_path} illustrates simulated sample paths of two independent GSTFCPs for different parameter values along with the corresponding GSTFSP-I sample path which is obtained as their difference. 
Since the two independent counting processes are driven by independent random time changes, their sample paths evolve independently before being combined to form the Skellam process sample path. 

Note that GSTFSP-II is obtained by applying an independent random time change to the underlying GSP which implies the upward and downward jump components of GSTFSP-II sample path share the same random clock. However, these components are obtained from two distinct independent time-changed counting processes (GSTFCPs) in GSTFSP-I due to which the random clocks for the components are different. Thus the stochastic dependence structure between the upward and downward components in GSTFSP-I and GSTFSP-II sample paths vary resulting in distinct patterns. From the above discussion, it's clear that the sample paths of GSTFSP-II can not be simulated using the algorithm for that of GSTFSP-I.

Finally for convenience, the following table provides a summary (abbreviations, names and notations) of the processes mentioned in this paper. The relations among these processes are discussed in Section \ref{special cases_GSTFCP}. 
\begin{table}[H]
\centering
\label{tab:abbreviations}
\begin{tabular}{lll}
\toprule
\textbf{Abbreviation} & \textbf{Process name} & \textbf{Notation} \\
\midrule
GSTFCP & generalized space time fractional counting process & $ M_{\beta}^{\alpha}(t)$ \\
GSFCP & generalized space fractional counting process & $ M_{\beta}(t)$ \\
GFCP & generalized fractional counting process & $ M^{\alpha}(t)$ \\
GCP & generalized counting process & $ M(t)$ \\
GSTFSP-I & generalized space-time fractional Skellam process of type~I & $\mathcal{S}_{\beta_1,\beta_2}^{\alpha_1,\alpha_2}(t)$\\
GSTFSP-II & generalized 
space-time fractional Skellam process of type~II & $\mathcal{S}_{\beta}^{\alpha}(t)$\\
GSFSP-I & generalized space fractional Skellam process of type~I & $\mathcal{S}_{\beta_1,\beta_2}(t)$\\
GSFSP-II & generalized 
space fractional Skellam process of type~II & $\mathcal{S}_{\beta}(t)$\\
GFSP-I & generalized fractional Skellam process of type~I & $\mathcal{S}^{\alpha_1,\alpha_2}(t)$\\
GFSP-II & generalized fractional Skellam process of type~II & $\mathcal{S}^{\alpha}(t)$\\
GSP & generalized Skellam process & $\mathcal{S}(t)$\\
STFSP-I & space-time fractional Skellam process of type~I & $S_{\beta_1,\beta_2}^{\alpha_1,\alpha_2}(t)$\\
STFSP-II & space-time fractional Skellam process of type~II & $S_{\beta}^{\alpha}(t)$\\
SFSP-I & space fractional Skellam process of type~I & $S_{\beta_1,\beta_2}(t)$\\
SFSP-II & space fractional Skellam process of type~II & $S_{\beta}(t)$\\
FSP-I & fractional Skellam process of type~I & $S^{\alpha_1,\alpha_2}(t)$\\
FSP-II & fractional Skellam process of type~II & $S^{\alpha}(t)$\\
STFSPoK-I & space-time fractional Skellam process of order $k$ of type~I & $S^{k,\alpha_1,\alpha_2}_{\beta_1,\beta_2}(t)$\\
STFSPoK-II & space-time fractional Skellam process of order $k$ of type~II & $S^{k,\alpha}_{\beta}(t)$\\
SFSPoK-I & space fractional Skellam process of order $k$ of type~I & $S^{k}_{\beta_1,\beta_2}(t)$\\
SFSPoK-II & space fractional Skellam process of order $k$ of type~II & $S^{k}_{\beta}(t)$\\
FSPoK-I & fractional Skellam process of order $k$ of type~I & $S^{k,\alpha_1,\alpha_2}(t)$\\
FSPoK-II & fractional Skellam process of order $k$ of type~II & $S^{k,\alpha}(t)$\\
SPoK & Skellam process of order $k$ & $S^{k}(t)$\\
\bottomrule
\end{tabular}
\end{table}

\begin{table}[H]
\centering
\begin{tabular}{lll}
\toprule
\bottomrule
\end{tabular}
\end{table}

\section*{Declarations}
\subsection*{Conflict of Interest} The authors declare that they have no conflict of interest.
\subsection*{Data availability} No data was used for the research described in this article.
\subsection*{Ethics Approval} Not applicable for the research described in this article. 
\subsection*{Funding} The authors received no specific funding for this research.

\bibliographystyle{plainnat}
\bibliography{GSTFSP}

\end{document}